\documentclass[12pt]{article}

\pdfoutput=1

\usepackage[english]{babel}
\usepackage[utf8]{inputenc}%coding of writteninput
\usepackage[T1]{fontenc}%vectorized fonts (cm-super package)

\usepackage{amsfonts}
\usepackage{amsmath, amsthm, amssymb}
\usepackage{graphicx}

\usepackage{enumitem}

\usepackage[margin=3cm]{geometry}

\newtheorem{assumption}{Assumption}
\theoremstyle{plain}

\newtheorem{theorem}{Theorem}
\theoremstyle{plain}

\newtheorem{remark}{Remark}
\theoremstyle{plain}

\newtheorem{definition}{Definition}
\theoremstyle{plain}

\newcommand{\hp}{s}

\newtheoremstyle{myAlgoStyle}
{}                 %Space above
{}                 %Space below
{\tt}                         %Body font: original {\normalfont}
{}                         %Indent amount (empty = no indent,%\parindent = paraindent)
{\normalfont\bfseries}  %Thm head font original
{ }  	%punctuation after head
{ }%Space after thm head: " " = normal interword %space; \newline = linebreak
{\textbf{\thmname{#1}\thmnumber{ #2} \thmnote{(#3)}}}
\theoremstyle{myAlgoStyle}
\newtheorem{myAlgorithm}{Algorithm}

\title{A stable and linear time discretization for a thermodynamically consistent model for
two-phase incompressible flow%
\footnote{The authors gratefully acknowledge the financial support by the Deutsche
Forschungsgemeinschaft through the priority program SPP1506 entitled ''Transport processes at
fluidic interfaces''.}}

\author{
Harald Garcke
\footnote{
Fakult\"at f\"ur Mathematik,
Universit\"at Regensburg,
93040 Regensburg.},
Michael Hinze
\footnote{
Fachbereich Mathematik,
Universit\"at Hamburg,
Bundesstra\ss e 55,
20146 Hamburg.}, 
Christian Kahle
\footnote{
Fachbereich Mathematik,
Universit\"at Hamburg,
Bundesstra\ss e 55, 
20146 Hamburg.}
}

\begin{document}

\maketitle

\begin{abstract}
A new time discretization scheme for the numerical simulation of two-phase flow governed by a
thermodynamically consistent diffuse interface model is presented.
The scheme is consistent in the sense that it allows for a discrete in
time energy inequality.
An adaptive spatial discretization is proposed that conserves the energy
inequality in the fully discrete setting by applying a suitable post processing step to the adaptive
cycle. For the fully discrete scheme a quasi-reliable error estimator is derived  which
estimates  the error both of the flow velocity, and of the phase field.  The
validity of the energy inequality in the fully discrete setting is numerically investigated.
\end{abstract}

\section*{Introduction}
In the present work we propose a stable and (essentially) linear  time discretization scheme for
two-phase flows governed by the diffuse interface model
\begin{align}
  \rho\partial_t v + \left( \left( \rho v + J\right)\cdot\nabla
  \right)v
  - \mbox{div}\left(2\eta Dv\right) + \nabla p =& \mu\nabla \varphi + \rho g &&
  \forall x\in\Omega,\, \forall t \in I,\label{eq:CHNSstrong1}\\
  \mbox{div}(v) = &0&&
  \forall x\in\Omega,\, \forall t \in I,\label{eq:CHNSstrong2}\\
  \partial_t \varphi + v \cdot\nabla \varphi - \mbox{div}(m\nabla \mu) = &0&&
  \forall x\in\Omega,\, \forall t \in I,\label{eq:CHNSstrong3}\\
  -\sigma\epsilon \Delta \varphi + F'(\varphi) - \mu =& 0&&
  \forall x\in\Omega,\, \forall t \in I,\label{eq:CHNSstrong4}\\
  v(0,x) =& v_0(x)      &&    \forall x \in \Omega,\label{eq:CHNSstrongIC1}\\
  \varphi(0,x) =& \varphi_0(x) &&\forall x \in \Omega,\label{eq:CHNSstrongIC2}\\
  v(t,x) =& 0 &&\forall x \in \partial \Omega,\, \forall t \in I,\label{eq:CHNSstrongBC1}\\
  \nabla \mu(t,x)\cdot \nu_\Omega =
  \nabla \varphi(t,x) \cdot \nu_\Omega =& 0  &&\forall x \in \partial
  \Omega,\, \forall t \in I,\label{eq:CHNSstrongBC2}
\end{align}
where $J = -\frac{d\rho}{d\varphi}m\nabla \mu$.
This model is proposed in \cite{AbelsGarckeGruen_CHNSmodell}.
Here $\Omega \subset \mathbb{R}^n,\, n\in \{2,3\}$, denotes an open and bounded domain, $I=(0,T]$
with $0<T<\infty$ a time interval,
$\varphi$ denotes the phase field,
$\mu$ the chemical potential,
$v$ the volume averaged velocity,
$p$ the pressure,
and $\rho = \rho(\varphi) = \frac{1}{2}\left((\rho_2-\rho_1)\varphi +(\rho_1+\rho_2)\right)$ the mean density,
where $0<\rho_1\leq\rho_2$ denote the densities of the involved fluids.
The viscosity is denoted by $\eta$ and can be chosen arbitrarily, fulfilling $\eta(-1) =
\tilde\eta_1$ and $\eta(1) = \tilde \eta_2$, with individual fluid viscosities $\eta_1,\eta_2$.
The mobility is denoted by $m = m(\varphi)$. The graviational force is denoted by $g$.
By $Dv = \frac{1}{2}\left(\nabla v + (\nabla v)^t\right)$ we denote the symmetrized gradient.
The scaled surface tension is denoted by $\sigma$ and the interfacial
width is proportional to $\epsilon$.
The free energy is denoted by $F$.
For $F$ we use a splitting $F = F_+ + F_-$, where $F_+$ is convex and $F_-$
is concave.

The above model couples the Navier--Stokes equations \eqref{eq:CHNSstrong1}--\eqref{eq:CHNSstrong2}
to the Cahn--Hilliard model \eqref{eq:CHNSstrong3}--\eqref{eq:CHNSstrong4} in a thermodynamically
consistent way, i.e. a free energy inequality holds.
It is the main goal to introduce and analyze an (essentially) linear time discretization scheme for
the numerical treatment of  \eqref{eq:CHNSstrong1}--\eqref{eq:CHNSstrongBC2}, which also on the
discrete level fulfills the free energy inequality. This in conclusion leads to a stable scheme
that is thermodynamically consistent on the discrete level.

Existence of weak solutions to system \eqref{eq:CHNSstrong1}--\eqref{eq:CHNSstrongBC2} for a
specific class of free energies $F$ is shown in
\cite{AbelsDepnerGarcke_CHNS_AGG_exSol,AbelsDepnerGarcke_CHNS_AGG_exSol_degMob}.
See also the work \cite{Gruen_convergence_stable_scheme_CHNS_AGG}, where the existence of
weak solutions for a different class of free energies $F$ is shown by passing to the limit in a
numerical scheme. 
We refer to \cite{Lowengrub_CahnHilliard_and_Topology_transitions},
\cite{Boyer_two_phase_different_densities},
\cite{Ding_Spelt_Shu_diffuse_interface_model_diff_density},
\cite{AkiDreyer_QuasiIncompressibleInterfaceModel}, 
and the review \cite{AndersenFaddenWheeler} for other diffuse interface models for two-phase
incompressible flow.
Numerical approaches for different variants of the Navier--Stokes Cahn--Hilliard system have been
studied in \cite{KayStylesWelford}, \cite{Feng_FullyDiscreteNSCH},
\cite{Boyer_two_phase_different_densities}, \cite{Aland_Voigt_bubble_benchmark},
\cite{Gruen_convergence_stable_scheme_CHNS_AGG}, \cite{HintermuellerHinzeKahle_adaptiveCHNS},
\cite{Gruen_Klingbeil_CHNS_AGG_numeric} and \cite{GuoLinLowengrub_numericalMethodForCHNS_Lowengrub}.

This work is organized as follows.
In Section \ref{sec:weak_and_TimeDisc} we derive a weak
formulation of \eqref{eq:CHNSstrong1}--\eqref{eq:CHNSstrongBC2} and formulate a time discretization
scheme.
In Section \ref{sec:fullyDiscrete} we derive the fully discrete model and
show the existence of solutions for both the time discrete, and the
fully discrete model, as well as energy inequalities, both for the time discrete model, and for the
fully discrete model.
In Section \ref{sec:Adaptivity} we use the energy inequality to derive a residual based adaptive
concept, and
in Section \ref{sec:Numerics} we numerically investigate properties of
our simulation scheme.

%%%%%%%%%%%%%%%%%%%%%%%%%%%%%%%%%%%%%%%%%%%%%%%%%%%%%%%%%%%%%%%%%%%%%%%%%%%%%%%%%%%
%%%%%%%%%%%%%%%%%%%%%%%%%%%%%%%%%%%%%%%%%%%%%%%%%%%%%%%%%%%%%%%%%%%%%%%%%%%%%%%%%%%
%%%%%%%%%%%%%%%%%%%%%%%%%%%%%%%%%%%%%%%%%%%%%%%%%%%%%%%%%%%%%%%%%%%%%%%%%%%%%%%%%%%
%%%%%%%%%%%%%%%%%%%%%%%%%%%%%%%%%%%%%%%%%%%%%%%%%%%%%%%%%%%%%%%%%%%%%%%%%%%%%%%%%%%
%as you expect
\section*{Notation and assumptions}

Let $\Omega \subset \mathbb{R}^n$, $n\in \{2,3\}$ denote a bounded domain with boundary
$\partial \Omega$ and outer normal $\nu_\Omega$.
Let $I = (0,T)$ denote a time interval.

We use the conventional notation for Sobolev and Hilbert Spaces, see e.g.
\cite{Adams_SobolevSpaces}.
With $L^p(\Omega)$, $1\leq p\leq \infty$, we denote the space of measurable functions on $\Omega$,
whose modulus to the power $p$ is Lebesgue-integrable. $L^\infty(\Omega)$ denotes the space of
measurable functions on $\Omega$, which are essentially bounded.
For $p=2$ we denote by
$L^2(\Omega)$ the space of square integrable functions on $\Omega$
with inner product $(\cdot,\cdot)$ and norm $\|\cdot \|$.
For a subset $D\subset \Omega$ and  functions $f,g\in L^2(\Omega)$ we by $(f,g)_D$  denote the inner product of
$f$ and $g$ restricted to $D$, and by $\|f\|_D$ the respective norm.
By $W^{k,p}(\Omega)$, $k\geq 1, 1\leq p\leq \infty$, we denote the Sobolev space of functions
admitting weak derivatives up to order $k$ in $L^p(\Omega)$.
If $p=2$ we write $H^k(\Omega)$.
The subset $H^1_0(\Omega)$ denotes $H^1(\Omega)$ functions
with vanishing boundary trace.\\
We further set
\begin{align*}
  L^2_{0}(\Omega) = \{v\in L^2(\Omega)\,|\, (v,1) = 0\},
\end{align*}
and with
\begin{align*}
  H(\mbox{div},\Omega) = \{v\in H^1_0(\Omega)^n\,|\, (\mbox{div}(v),q) = 0\,\forall q\in
  L^2_{(0)}(\Omega) \}
\end{align*}
we denote the space of all weakly solenoidal $H^1_0(\Omega)$ vector fields.\\
For $u\in L^q(\Omega)^n$, $q>n$, and $v,w \in H^1(\Omega)^n$ we introduce the
trilinear form
\begin{align}
  a(u,v,w) =
  \frac{1}{2}\int_\Omega\left( \left(u\cdot\nabla\right)v \right)w\,dx
  -\frac{1}{2}\int_\Omega\left( \left(u\cdot\nabla\right)w \right)v\,dx. \label{eq:trilinearForm}
\end{align}
Note that there holds $a(u,v,w) = -a(u,w,v)$, and especially $a(u,v,v) = 0$.

\medskip

\noindent For the data of our problem we assume:
\begin{enumerate}[label=A\arabic{*}]
  \item \label{ass:rhoetamob_bounded}
There exists constants $\overline \rho \geq \underline \rho>0$,
$\overline \eta \geq \underline \eta>0$, and
$\overline m \geq \underline m>0$ such that the following relations are
satisfied:
  \begin {itemize}
    \item[$\bullet$] $\overline \rho \geq \rho(\varphi)\geq\underline \rho>0$,
    \item[$\bullet$] $\overline \eta \geq \eta(\varphi)\geq\underline \eta>0$,
    \item[$\bullet$] $\overline m\geq m(\varphi)\geq\underline m>0$.
  \end{itemize}
  Especially we assume that the mobility is non degenerated. In addition we assume, that $\rho$,
  $\mu$, and $m$ are continuous.
  \item \label{ass:F_C1}
  $F:\mathbb{R}\to \mathbb{R}$ is continuously differentiable.
  \item \label{ass:F_boundedPoly}
  $F$ and the derivatives $F'_+$ and $F'_-$ are polynomially bounded, i.e. there exists $C>0$
  such that $|F(x)| \leq C(1+|x|^q)$, $|F'_+(x)| \leq C(1+|x|^{q-1})$ 
  and $|F'_-(x)| \leq  C(1+|x|^{q-1})$ holds for some
  $q \in [1,4]$ if $n=3$ and $q \in [1,\infty)$ if $n=2$, 
  \item \label{ass:F_NewtonDiff}
  $F'_+$ is Newton (sometimes called slantly) differentiable (see e.g.
  \cite{Hintermueller_primaldualSet}) regarded as nonlinear operator $F'_+: H^1(\Omega)\to
  \left(H^1(\Omega)\right)^*$ with Newton derivative $G$ satisfying
  \begin{align*}
    (G(\varphi)\delta\varphi,\delta\varphi)\geq 0
  \end{align*}
   for each $\varphi\in
  H^1(\Omega)$ and $\delta \varphi\in H^1(\Omega)$.
\end{enumerate} 

To ensure Assumption \ref{ass:rhoetamob_bounded} we introduce
a cut-off mechanism to ensure the bounds on $\rho$ defined in Assumption
\ref{ass:rhoetamob_bounded} independently of $\varphi$. Note that $\eta(\varphi)$ and $m(\varphi)$
can be chosen arbitrarily fulfilling the stated bounds.
We define the mass density as a smooth, monotone and strictly positive function  $\rho(\varphi)$
fulfilling
\begin{align*}
  \rho(\varphi) = 
  \begin{cases}
    \frac{\tilde \rho_2-\tilde \rho_1}{2}\varphi + \frac{\tilde\rho_1+\tilde\rho_2}{2}
    & \mbox{ if } -1-\frac{\tilde \rho_1}{\tilde\rho_2-\tilde\rho_1} 
    < \varphi< 
    1+\frac{\tilde \rho_1}{\tilde\rho_2-\tilde\rho_1},\\
    \mbox{const} & 
    \mbox{ if } \varphi > 1 + \frac{2\tilde\rho_1}{\tilde\rho_2-\tilde\rho_1},\\
    \mbox{const} & 
    \mbox{ if } \varphi < -1 - \frac{2\tilde \rho_1}{\tilde\rho_2-\tilde\rho_1}.
  \end{cases}
\end{align*} 
For a discussion we refer to \cite[Remark 2.1]{Gruen_convergence_stable_scheme_CHNS_AGG}.

\begin{remark}\label{rm:freeEnergies}
The Assumptions \ref{ass:F_C1}--\ref{ass:F_NewtonDiff} are for example fulfilled by the
polynomial free energy 
\begin{align*}
  F^{poly}(\varphi) = \frac{\sigma}{4\epsilon}\left(1-\varphi^2\right)^2.
\end{align*}
Another free energy fulfilling these assumptions is the relaxed double-obstacle free energy
given by
\begin{align} 
   F^{rel}(\varphi) & = \frac{\sigma}{2\epsilon}\left(1-\varphi^2+\hp\lambda^2( \varphi)\right),
  \label{eq:freeEnergy}
\end{align}
with 
\begin{align*}
    \lambda(\varphi)& := \max(0,\varphi-1) + \min(0,\varphi+1),
\end{align*}
where $\hp\gg 0$ denotes the relaxation parameter.
$F^{rel}$ is introduced in \cite{HintermuellerHinzeTber} as Moreau--Yosida relaxation of the
double-obstacle free energy
\begin{align*}
  F^{obst}(\varphi) = 
  \begin{cases}
    \frac{\sigma}{2\epsilon}\left(1-\varphi^2\right) & \mbox{ if } |\varphi|\leq 1,\\
    0           & \mbox{ else}, 
  \end{cases}
\end{align*}
which is proposed in \cite{BloweyElliott_I} to model phase separation.

In the numerical examples of this work we use the free energy $F \equiv F^{rel}$.
 For this choice the splitting into convex and concave part reads
\begin{align*}
  F_+(\varphi) &= \hp\frac{\sigma}{2\epsilon}\lambda^2(\varphi),
  &
  F_-(\varphi) &= \frac{\sigma}{2\epsilon}(1-\varphi^2).
\end{align*}
\end{remark}

%%%%%%%%%%%%%%%%%%%%%%%%%%%%%%%%%%%%%%%%%%%%%%%%%%%%%%%%%%%%%%%%%%%%%%%%%%%%%%%%%%%
%%%%%%%%%%%%%%%%%%%%%%%%%%%%%%%%%%%%%%%%%%%%%%%%%%%%%%%%%%%%%%%%%%%%%%%%%%%%%%%%%%%
%%%%%%%%%%%%%%%%%%%%%%%%%%%%%%%%%%%%%%%%%%%%%%%%%%%%%%%%%%%%%%%%%%%%%%%%%%%%%%%%%%%
%%%%%%%%%%%%%%%%%%%%%%%%%%%%%%%%%%%%%%%%%%%%%%%%%%%%%%%%%%%%%%%%%%%%%%%%%%%%%%%%%%%
%Here we derive the weak formulation as well as the time discretization
%section identifier in labels: TD
\section{The time discrete setting}\label{sec:weak_and_TimeDisc}
In the present section we formulate our time discretization scheme
that is based on a weak formulation of
\eqref{eq:CHNSstrong1}--\eqref{eq:CHNSstrongBC2}
which we derive next.
To begin with,  note that for a sufficiently smooth solution
$(\varphi,\mu,v)$  of
\eqref{eq:CHNSstrong1}--\eqref{eq:CHNSstrongBC2}
we can rewrite \eqref{eq:CHNSstrong1},
using the linearity of $\rho$, as
\begin{align}
  \partial_t( \rho v) + \mbox{div}\left( \rho v \otimes v \right)
  + \mbox{div}\left(v\otimes  J\right) - \mbox{div}\left(2\eta Dv\right) + \nabla p = \mu\nabla
  \varphi + \rho g,\label{eq:CHNSstrong1-2}
\end{align}
see \cite[p. 14]{AbelsGarckeGruen_CHNSmodell}.

We also note that the term $\rho v + J$ in \eqref{eq:CHNSstrong1} is not solenoidal
(which might lead to difficulties both in the analytical and the numerical treatment)
and that the trilinear form $(((\rho v +J)\cdot \nabla) u,w)$ is not anti-symmetric.
To obtain a weak formulation yielding an anti-symmetric convection term we use a convex combination
of  \eqref{eq:CHNSstrong1} and \eqref{eq:CHNSstrong1-2} to define a weak formulation.
We multiply equations \eqref{eq:CHNSstrong1} and \eqref{eq:CHNSstrong1-2} by the solenoidal  test
function $\frac{1}{2}w \in H( \mbox{div},\Omega)$, integrate over $\Omega$, add the resulting
equations and perform integration by parts.
This gives
\begin{align*}
  \frac{1}{2}\int_{\Omega}\left( \partial_t(\rho v)+\rho\partial_t v \right) w\,dx
  +\int_{\Omega}  2\eta Dv:Dw\,dx+a(\rho v +J,v,w)
  =   \int_\Omega \mu\nabla \varphi w + \rho g w\,dx.
\end{align*}
Equations \eqref{eq:CHNSstrong3}--\eqref{eq:CHNSstrong4} are treated classically.
This leads to
\begin{definition}
We call $v$, $\varphi$, $\mu$ a
weak solution to \eqref{eq:CHNSstrong1}--\eqref{eq:CHNSstrongBC2}
if $v(0) = v_0$, $ \varphi(0) = \varphi_0$, $v(t) \in H(\mbox{div},\Omega)$ for $a.e.\,t \in I$ and
\begin{align}
  \frac{1}{2}\int_{\Omega}\left( \partial_t(\rho v)+\rho\partial_t v \right) w\,dx
  +\int_{\Omega}2\eta Dv:Dw\,dx&\nonumber\\
  +a(\rho v + J,v,w)
  =   \int_\Omega \mu\nabla \varphi w +\rho g w\,dx &\quad \forall w\in H(\mbox{div},\Omega),
  \label{eq:CHNS1_weak}\\
  \int_\Omega\left(\partial_t\varphi + v\cdot \nabla \varphi\right) \Phi\,dx
  + \int_\Omega m(\varphi)\nabla \mu \cdot \nabla \Phi\,dx =0
  &\quad \forall \Phi \in  H^1(\Omega),
  \label{eq:CHNS2_weak}\\
  \sigma \epsilon\int_\Omega\nabla \varphi\cdot\nabla \Psi\,dx
  +\int_\Omega F'(\varphi)\Psi\,dx
  - \int_\Omega \mu\Psi\,dx = 0 &\quad \forall \Psi \in H^1(\Omega),
  \label{eq:CHNS3_weak}
\end{align}
is satisfied for almost all $t\in I$.
\end{definition}

\begin{theorem}\label{thm:ContinousEnergyEstimate}
Let $v,\varphi,\mu$ be a sufficiently smooth
 solution to
\eqref{eq:CHNS1_weak}--\eqref{eq:CHNS3_weak}.
Then  there holds 
\begin{align*}
  \frac{1}{2}\frac{d}{dt}\left( \int_\Omega \rho |v|^2 + \sigma\epsilon |\nabla \varphi|^2 +
  F(\varphi)\,dx\right) = -\int_\Omega 2\eta |Dv|^2 + m|\nabla \mu|^2\,dx 
  + \int_\Omega  \rho g  v\,dx .
  \end{align*}
\end{theorem} 
\begin{proof}
By testing \eqref{eq:CHNS1_weak} with $w\equiv v$, \eqref{eq:CHNS2_weak} with $\Phi \equiv \mu$ and
\eqref{eq:CHNS3_weak} with $\Psi \equiv \partial_t \varphi$ and adding the resulting equations the
claim follows.
\end{proof}

In \cite{AbelsDepnerGarcke_CHNS_AGG_exSol,AbelsDepnerGarcke_CHNS_AGG_exSol_degMob}
an alternative weak formulation of \eqref{eq:CHNSstrong1}--\eqref{eq:CHNSstrongBC2} is proposed, for
which the authors show existence of weak solutions.

We now introduce a time discretization which mimics the energy inequality in Theorem
\ref{thm:ContinousEnergyEstimate} on the discrete level.
 Let
$0=t_0<t_1<\ldots<t_{k-1}<t_k<t_{k+1}<\ldots<t_M=T$
denote an equidistant subdivision of the
interval $\overline I = [0,T]$ with $\tau_{k+1}-\tau_k = \tau$.
From here onwards the superscript $k$ denotes the corresponding variables at time instance
$t_k$.

\bigskip

\noindent \textbf{Time integration scheme}\\
Let $\varphi_0 \in H^1(\Omega)$
% with $|\varphi_0|\leq 1 \,a.e.$
and $v_0 \in  H(\mbox{div},\Omega)$.

\medskip

\noindent \textit{Initialization for $k=0$:}\\
Set $\varphi^0 = \varphi_0$ and $v^0=v_0$.\\
Find $\varphi^1 \in H^1(\Omega)$, $\mu^1\in H^1(\Omega)$, $v^1 \in H(\mbox{div},\Omega)$, such that
for all $w\in H(\mbox{div},\Omega)$, $\Phi \in H^1(\Omega)$, and $\Psi \in H^1(\Omega)$ it holds
\begin{align}
  \frac{1}{\tau}\int_\Omega \rho^1(v^1-v^0) w \,dx
  +\int_\Omega ((\rho^0v^0+J^1)\cdot \nabla) v^1 \cdot w\,dx&\nonumber\\
  +\int_{\Omega} 2\eta^1  Dv^{1}:Dw\,dx
  -\int_\Omega \mu^{1}\nabla \varphi^1 w + \rho^1 g w\,dx  &= 0&&
  \forall w \in H(\mbox{div},\Omega),\label{eq:TD:chns1_solenoidal_init}\\
  \frac{1}{\tau}\int_\Omega (\varphi^{1}-\varphi^0) \Phi \,dx +
  \int_\Omega(v^{0}\cdot \nabla \varphi^{0}) \Phi \, dx& \nonumber \\
  +  \int_\Omega m(\varphi^0)\nabla \mu^{1}\cdot\nabla \Phi\,dx &=0 &&\forall
  \Phi\in H^1(\Omega),\label{eq:TD:chns2_solenoidal_init}\\
  \sigma \epsilon\int_\Omega\nabla \varphi^{1}\cdot\nabla \Psi\,dx
  - \int_\Omega \mu^{1}\Psi\,dx \nonumber \\
  +\int_\Omega ((F_+)'(\varphi^{1})+(F_-)'(\varphi^0))\Psi\,dx &= 0&&\forall
  \Psi\in H^1(\Omega),\label{eq:TD:chns3_solenoidal_init}
\end{align}
where $ J^1 := -\frac{d\rho}{d\varphi}(\varphi^1)m^1\nabla \mu^1$.

\medskip

\noindent \textit{Two-step scheme for $k\geq 1$:} \\
Given $\varphi^{k-1}\in H^1(\Omega)$,
$\varphi^k\in H^1(\Omega)$,
$\mu^k \in W^{1,q}(\Omega)$, $q>n$,
$v^k \in H(\mbox{div},\Omega)$,\\
find
$v^{k+1} \in H(\mbox{div},\Omega)$, $\varphi^{k+1}\in H^1(\Omega)$, $\mu^{k+1}\in
H^{1}(\Omega)$ satisfying
\begin{align}
  \frac{1}{2\tau}\int_\Omega \left( \rho^kv^{k+1}-\rho^{k-1}v^{k} \right) w +
  \rho^{k-1}(v^{k+1}-v^k)w\,dx&\nonumber\\
  +a(\rho^k v^k+J^k,v^{k+1},w)+\int_{\Omega} 2\eta^k Dv^{k+1}:Dw\,dx\nonumber\\
  -   \int_\Omega \mu^{k+1}\nabla \varphi^k w - \rho^{k} g w\,dx &= 0&&
  \forall w \in H(\mbox{div},\Omega),\label{eq:TD:chns1_solenoidal}\\
  \frac{1}{\tau}\int_\Omega (\varphi^{k+1}-\varphi^k) \Phi \,dx +
  \int_\Omega(v^{k+1}\cdot \nabla \varphi^k) \Phi \, dx& \nonumber \\
  +  \int_\Omega m(\varphi^k)\nabla \mu^{k+1}\cdot\nabla \Phi\,dx &=0 &&\forall
  \Phi\in H^1(\Omega),\label{eq:TD:chns2_solenoidal}\\
  \sigma \epsilon\int_\Omega\nabla \varphi^{k+1}\cdot\nabla \Psi\,dx
  - \int_\Omega \mu^{k+1}\Psi\,dx \nonumber \\
  +\int_\Omega ((F_+)'(\varphi^{k+1})+(F_-)'(\varphi^k))\Psi\,dx &= 0&&\forall
  \Psi\in H^1(\Omega),\label{eq:TD:chns3_solenoidal}
\end{align}
where $ J^k := -\frac{d\rho}{d\varphi}(\varphi^k)m^k\nabla \mu^k$.

We note that in \eqref{eq:TD:chns1_solenoidal}--\eqref{eq:TD:chns3_solenoidal} the only nonlinearity
arises from $F_+'$ and thus only the equation \eqref{eq:TD:chns3_solenoidal} is nonlinear.
Let us summarize properties of this scheme in the following remark.

\begin{remark}\label{rm:InitializationStep}
~
\begin{itemize}
  \item The time discretization \eqref{eq:TD:chns1_solenoidal_init}--\eqref{eq:TD:chns3_solenoidal_init}
  used in the  initialization step is motivated by the time discretization
  in \cite{KayStylesWelford} for the equal density case.
  In particular it yields a sequential coupling of the Cahn--Hilliard and the Navier--Stokes
  systems. 
  Concerning the
  existence of a unique solution we refer to e.g. \cite{HintermuellerHinzeKahle_adaptiveCHNS}.
  From the regularity theory for the Laplace operator we have $\mu^1\in H^2(\Omega)$.
  \item Existence and uniqueness of a solution to the time discrete model
  \eqref{eq:TD:chns1_solenoidal}--\eqref{eq:TD:chns3_solenoidal} is shown in Theorem
  \ref{thm:TD:exSol}.
  Using the Assumption \ref{ass:F_NewtonDiff} posed on $F$, it can be shown that
  Newton's method in function space can be used to compute a solution to
  \eqref{eq:TD:chns1_solenoidal}--\eqref{eq:TD:chns3_solenoidal} using the steps from Theorem
  \ref{thm:TD:exSol}.
  \item Through the use of $\rho^{k-1}$,
  \eqref{eq:TD:chns1_solenoidal}--\eqref{eq:TD:chns3_solenoidal} is a 2-step scheme. However,
  by replacing \eqref{eq:TD:chns1_solenoidal} with
  \begin{align*}
    \frac{1}{2\tau}\int_\Omega \left( \rho^{k+1}v^{k+1}-\rho^{k}v^{k} \right) w +
    \rho^{k}(v^{k+1}-v^k)w\,dx\\
    +a(\rho^k v^k+J^k,v^{k+1},w)+\int_{\Omega} 2\eta Dv^{k+1}:Dw\,dx\\
    -   \int_\Omega \mu^{k+1}\nabla \varphi^k w + \rho^{k}g w\,dx &= 0\quad
    \forall w \in H(\mbox{div},\Omega),
  \end{align*}
  one obtains an one-step scheme, which then also is nonlinear in the time discretization of
  \eqref{eq:CHNS1_weak}.
  The resulting system is analyzed in a forthcoming paper.
\end{itemize}
\end{remark}

In \cite{Gruen_Klingbeil_CHNS_AGG_numeric} Gr\"un and Klingbeil propose a time-discrete solver
for \eqref{eq:CHNSstrong1}--\eqref{eq:CHNSstrongBC2} which leads to strongly coupled systems for
$v,\varphi$ and $p$ at every time step and requires a fully nonlinear solver. For this scheme Gr\"un
in \cite{Gruen_convergence_stable_scheme_CHNS_AGG} proves an energy inequality and the existence of
so called generalized solutions.

%%%%%%%%%%%%%%%%%%%%%%%%%%%%%%%%%%%%%%%%%%%%%%%%%%%%%%%%%%%%%%%%%%%%%%%%%%%%%%%%%%%
%%%%%%%%%%%%%%%%%%%%%%%%%%%%%%%%%%%%%%%%%%%%%%%%%%%%%%%%%%%%%%%%%%%%%%%%%%%%%%%%%%%
%%%%%%%%%%%%%%%%%%%%%%%%%%%%%%%%%%%%%%%%%%%%%%%%%%%%%%%%%%%%%%%%%%%%%%%%%%%%%%%%%%%
%%%%%%%%%%%%%%%%%%%%%%%%%%%%%%%%%%%%%%%%%%%%%%%%%%%%%%%%%%%%%%%%%%%%%%%%%%%%%%%%%%%
%Here we state the fully discrete model and show the fully discrete energy estimate
%We further show the existence of solution for both the fully discrete and the time discrete setting
% and therafter show the energy inequality in the semidiscrete setting
%section identifier in labels: FD
\section{The fully discrete setting and energy inequalities}\label{sec:fullyDiscrete}
For a numerical treatment we next discretize the weak formulation
\eqref{eq:TD:chns1_solenoidal}--\eqref{eq:TD:chns3_solenoidal} in space.
We aim at an adaptive discretization of the domain $\Omega$, and thus to have a different
spatial discretization in every time step.

Let $\mathcal T^{k} = \bigcup_{i=1}^{NT}T_i$ denote a conforming triangulation of
$\overline \Omega$ with closed simplices $T_i,i=1,\ldots,NT$ and edges $E_i,i=1,\ldots,NE$,
$\mathcal{E}^{k} = \bigcup_{i=1}^{NE}E_i$. Here $k$ refers to the time instance $t_{k}$.
On $\mathcal T^{k}$ we define the following finite element spaces:
\begin{align*}
  \mathcal{V}^{1}(\mathcal T^{k})
  =& \{v\in C(\mathcal{T}^{k}) \, |
  \, v|_T \in P^1(T)\, \forall T\in  \mathcal{T}^{k}\}
  =: \mbox{span}\{\Phi^i\}_{i=1}^{NP},\\
  \mathcal{V}^{2}(\mathcal T^{k}) =& \{v\in C(\mathcal{T}^{k}) \, |
  \, v|_T \in P^2(T)\, \forall T\in  \mathcal{T}^{k}\},
\end{align*}
where $P^l(S)$ denotes the space of polynomials up to order $l$ defined on $S$.

We introduce the discrete analogon to the space $H(\mbox{div},\Omega)$:
\begin{align*}
  H(\mbox{div},\mathcal{T}^{k}) &=
  \{ v\in \mathcal V^2(\mathcal T^{k})^n\,|\, (\mbox{div}v,q) =  0\,
  \forall q\in \mathcal V^1(\mathcal T^{k}) \cap L^2_{(0)}(\Omega),\, 
  v|_{\partial \Omega}=0\}\\
  &:= \mbox{span}\{b^i\}_{i=1}^{NF},
\end{align*}

We further  introduce a $H^1$-stable projection operator
$\mathcal{P}^{k} : H^1(\Omega) \to \mathcal{V}^1(\mathcal T^{k})$
satisfying
\begin{align*}
  \|\mathcal{P}^{k}v\|_{L^p(\Omega)} \leq \|v\|_{L^p(\Omega)}
\mbox{  and  }
  \|\nabla \mathcal{P}^{k}v \|_{L^r(\Omega)} \leq \|\nabla v\|_{L^r(\Omega)}
\end{align*}
for $v \in H^1(\Omega)$ with $r \in [1,2]$ and $p\in [1,6)$ if $n=3$, and $p\in [1,\infty)$ if
$n=2$.
Possible choices are the Cl\'ement operator (\cite{Clement_Interpolation}) or, by restricting
the preimage to $C(\overline \Omega) \cap H^1(\Omega)$, the Lagrangian interpolation operator.

\bigskip 

Using these spaces we state the discrete counterpart of
\eqref{eq:TD:chns1_solenoidal}--\eqref{eq:TD:chns3_solenoidal}:

Let $k\geq 1$, given
$\varphi^{k-1}\in \mathcal{V}^1(\mathcal{T}^{k-1})$,
$\varphi^{k}\in \mathcal{V}^1(\mathcal{T}^{k})$,
$\mu^{k}\in \mathcal{V}^1(\mathcal{T}^{k})$,
$v^{k}\in  H(\mbox{div},\mathcal T^{k})$,
find
$v^{k+1}_h \in  H(\mbox{div},\mathcal{T}^{k+1})$,
$\varphi^{k+1}_h\in \mathcal V^1(\mathcal  T^{k+1})$,
$\mu^{k+1}_h \in \mathcal V^1(\mathcal T^{k+1})$
such that for all
$w \in  H(\mbox{div},\mathcal T^{k+1})$,
$\Phi \in \mathcal V^1(\mathcal T^{k+1})$,
$\Psi \in  \mathcal V^1(\mathcal T^{k+1})$
there holds:
\begin{align}
  \frac{1}{2\tau}(\rho^{k}v^{k+1}_h-\rho^{k-1}v^k+\rho^{k-1}(v^{k+1}_h-v^k),w)
  + a(\rho^kv^k+J^k,v_h^{k+1},w)\nonumber\\
  +(2\eta^kDv^{k+1}_h,D w)-(\mu^{k+1}_h\nabla\varphi^{k}+\rho^k g,w)
  &= 0,\label{eq:FD:chns1_solenoidal}\\
  \frac{1}{\tau}(\varphi^{k+1}_h-\mathcal{P}^{k+1}\varphi^k,\Phi)+(m(\varphi^k)\nabla
  \mu^{k+1}_h,\nabla \Phi) +(v^{k+1}_h\nabla \varphi^k, \Phi)
  &=0,\label{eq:FD:chns2_solenoidal}\\
  \sigma\epsilon(\nabla \varphi^{k+1}_h,\nabla
  \Psi)+(F_+^\prime(\varphi^{k+1}_h)+F^\prime_-(\mathcal{P}^{k+1}\varphi^k),\Psi)-(\mu^{k+1}_h,\Psi)
  &=0, \label{eq:FD:chns3_solenoidal}
\end{align}
where
$\varphi^0 = P\varphi_0$ denotes the $L^2$ projection of $\varphi_0$ in 
$\mathcal{V}^1(\mathcal T^0)$, 
$v^0 = P^Lv_0$ denotes the Leray projection of $v_0$ in $H(\mbox{div},\mathcal T^0)$ (see
\cite{ConstantinFoias_NS}), and $\varphi_h^1,\mu_h^1,v_h^1$ are obtained from the fully discrete variant of
\eqref{eq:TD:chns1_solenoidal_init}--\eqref{eq:TD:chns3_solenoidal_init}.

\subsection{Existence of solution to the fully discrete system}\label{ssec:ExSol_FD}
We next show the existence of a unique solution to the fully discrete system
\eqref{eq:FD:chns1_solenoidal}--\eqref{eq:FD:chns3_solenoidal}.

\begin{theorem}\label{thm:FD:exSol}
There exist
$v^{k+1}_h \in H(\mbox{div},\mathcal T^{k+1})$,
$\varphi^{k+1}_h \in \mathcal{V}^1(\mathcal{T}^{k+1})$,
$\mu^{k+1}_h \in\mathcal{V}^1(\mathcal{T}^{k+1})$ solving
\eqref{eq:FD:chns1_solenoidal}--\eqref{eq:FD:chns3_solenoidal}.
\end{theorem}

\begin{proof}
By testing \eqref{eq:FD:chns2_solenoidal} with $\Phi\equiv 1$, integration by parts in
$(v^{k+1}_h\nabla\varphi^k,1)$ and using $v^{k+1}_h\in H(\mbox{div},\mathcal{T}^{k+1})$ we obtain
\begin{align*}
  (\varphi^{k+1}_h,1) = (\mathcal P^{k+1}\varphi^k,1).
\end{align*}
We define $\alpha=\frac{1}{|\Omega|}\int_\Omega \mathcal P^{k+1}\varphi^k\,dx$ and set
\begin{align*}
  V_{(0)} := \{ v_h\in \mathcal{V}^1(\mathcal{T}^{k+1})\, \mid \, (v_h,1)=0\}.
\end{align*}
Then $z^{k+1}:= \varphi^{k+1}-\alpha$ fulfills $z^{k+1}\in V_{(0)}$. In the following we use
$z^{k+1}$ as unknown for the phase field, since the mean value of $\varphi$ is fixed.
In addition we introduce $y^{k+1} := \mu^{k+1}_h-\frac{1}{|\Omega|}\int \mu^{k+1}_h\,dx$ and require
\eqref{eq:FD:chns2_solenoidal}--\eqref{eq:FD:chns3_solenoidal} preliminarily only for test functions
with zero mean value.

We define 
\begin{align*}
  X = H(\mbox{div},\mathcal{T}^{k+1}) \times V_{(0)} \times V_{(0)},
\end{align*}
with the inner product 
\begin{align*}
  \left( (v_1,y_1,z_1),(v_2,y_2,z_2) \right)_X := (Dv_1,Dv_2) + (\nabla y_1,\nabla y_2) + (\nabla
  z_1,\nabla z_2),
\end{align*}
and norm $\|\cdot\|_X^2 = (\cdot,\cdot)_X$.
It follows from the inequalities of Korn and Poincar\'e that $(\cdot,\cdot)_X$ indeed forms an inner
product on $X$.
For $(v,y,z)\in X$ we define
\begin{align*}
  \left(G(v,y,z),(\overline v,\overline y, \overline z)\right)_X := &
  \left(\frac{1}{2}(\rho^k+\rho^{k-1})v-\rho^{k-1}v^k,\overline v\right)
  +\tau a(\rho^kv^k+J^k,v,\overline v)\\
  &+ \tau (2\eta^k Dv,D\overline v)
  -\tau(y\nabla \varphi^k,\overline v) 
  - \tau(\rho^kg,\overline v)\\
  &+(z-\mathcal P^{k+1}\varphi^k,\overline y)
  + \tau(m(\varphi^k)\nabla y,\nabla \overline y)
  +\tau(v\nabla\varphi^k,\overline y)\\
  &+\sigma\epsilon(\nabla z,\nabla \overline z)
  +(F'_+(z+\alpha)+F'_-(\mathcal P^{k+1}\varphi^k),\overline z)
  -(y,\overline z).
\end{align*}
Now we show $(G(v,y,z),(v,y,z))_X >0$ for $\|(v,y,z)\|_X$ large enough and that $G$ satisfies the
supposition of \cite[Lem. II.1.4]{Temam_NavierStokes_old}. It then follows from 
\cite[Lem. II.1.4]{Temam_NavierStokes_old}, that $G$
 admits a root $(v^*,y^*,z^*) \in X$.

% The right hand side is a linear functional in $(\overline v, \overline y,\overline z)$ on the
% Hilbert space $X$ and hence the above representation follows from the Riesz representation theorem.
The function $G$ is obviously continuous.
We now estimate
\begin{equation}
\begin{aligned}
  (G(v,y,z),(v,y,z))_X \geq& \, \underline \rho (v,v) + 2\tau\underline \eta(Dv,Dv)\\
  & + \tau \underline m(\nabla y,\nabla y) + \sigma \epsilon(\nabla z,\nabla z)+(F'_+(z+\alpha),z)\\
  & -(\rho^{k-1}v^k,v) - \tau(\rho^k g,v)-(\mathcal P^{k+1}\varphi^k,y)+(F'_-(\mathcal
  P^{k+1}\varphi^k),z).
\end{aligned}
\label{eq:FD:proofExSol_G}
\end{equation}
Using the convexity of $F_+$, which implies that $F_+'$ is monotone, we obtain
\begin{align*}
  (F'_+(z+\alpha),z) = (F'_+(z+\alpha)-F'_+(\alpha),z) + (F'_+(\alpha),z) \geq (F'_+(\alpha),z).
\end{align*}
By using H\"older's and Poincar\'e's inequality in \eqref{eq:FD:proofExSol_G} we obtain 
\begin{align*}
  (G(v,y,z),(v,y,z))_X >0
\end{align*}
for $\|(v,y,z)\|_X\geq R$ if $R$ is large enough.
Now \cite[Lem. II.1.4]{Temam_NavierStokes_old} implies the existence of  $(v^*,y^*,z^*)\in X$ such
that $G(v^*,y^*,z^*)=0$.
Defining $(v,\mu,\varphi) = (v^*,y^*+\beta,z^*+\alpha)$ with $\beta$ such that $(\beta,1) =
(F'_+(\varphi)+F'_-(\mathcal P^{k+1}\varphi^k),1)$ we obtain that $(v,\mu,\varphi)$ solves
\eqref{eq:FD:chns1_solenoidal}--\eqref{eq:FD:chns3_solenoidal}.
\end{proof}

\begin{remark}
Note that we do not need that the variables from old time instances are defined on the mesh used on
the current time instance. We further do not need any smallness requirement on the mesh size $h$ or
on the  time step length $\tau$.
\end{remark}

\begin{theorem}\label{thm:FD:energyInequality}
Let $(\varphi^{k+1}_h$, $\mu_h^{k+1}$, $v_h^{k+1})$ be a solution to
\eqref{eq:FD:chns1_solenoidal}--\eqref{eq:FD:chns3_solenoidal}.
Then for $k\geq 1$:
\begin{align}
  \frac{1}{2}\int_\Omega \rho^k\left|v^{k+1}_h\right|^2\,dx
  + \frac{\sigma\epsilon}{2}\int_\Omega |\nabla \varphi^{k+1}_h|^2\,dx
  +\int_\Omega F(\varphi^{k+1}_h)\,dx\nonumber\\
  + \frac{1}{2}\int_\Omega \rho^{k-1}|v^{k+1}_h-v^k|^2\,dx
  + \frac{\sigma\epsilon}{2}
  \int_\Omega |\nabla \varphi^{k+1}_h-\nabla\mathcal{P}^{k+1} \varphi^k|^2\,dx\nonumber\\
  + \tau\int_\Omega 2\eta^k|Dv^{k+1}_h|^2\,dx
  +\tau\int_\Omega m^k|\nabla \mu^{k+1}_h|^2\,dx \nonumber\\
  \leq
  \frac{1}{2}\int_\Omega \rho^{k-1}\left|v^{k}\right|^2\,dx
  + \frac{\sigma\epsilon}{2}\int_\Omega |\nabla \mathcal{P}^{k+1}\varphi^{k}|^2\,dx
  +\int_\Omega F(\mathcal{P}^{k+1}\varphi^{k})\,dx 
  +\tau \int_\Omega  \rho^k g v^{k+1}_h.\label{eq:FD:energyInequality}
\end{align}
\end{theorem}
\begin{proof}
We have
\begin{align}
  \frac{1}{2}\left(\rho^kv_h^{k+1} - \rho^{k-1}v^k\right)\cdot v_h^{k+1} +
  \frac{1}{2}\rho^{k-1}\left(v_h^{k+1}-v^k\right)\cdot v_h^{k+1}\nonumber\\
  = \frac{1}{2}\rho^k\left|v_h^{k+1}\right|^2 + \frac{1}{2}\rho^{k-1}\left|v_h^{k+1}-v^k\right|^2
  -\frac{1}{2}\rho^{k-1}\left|v^k\right|^2,\label{eq:FD:rho_v_ineq}
\end{align}
\begin{align}
  \nabla \varphi_h^{k+1}&\cdot \left(\nabla \varphi_h^{k+1}-\nabla \varphi^k\right)\nonumber \\
  = &\frac{1}{2}|\nabla \varphi_h^{k+1}|^2 - \frac{1}{2}|\nabla \varphi^k|^2
  + \frac{1}{2}|\nabla \varphi_h^{k+1}-\nabla \varphi^k|^2, \label{eq:FD:varphi_ineq}
\end{align}
and since $F_+$ is convex and $F_-$ is concave,
\begin{align}
  F_+(\varphi_h^{k+1}) - F_+(\varphi^k) &\leq
  F'_+(\varphi_h^{k+1})(\varphi_h^{k+1}-\varphi^k),\label{eq:FD:Fplus_ineq}\\
  F_-(\varphi_h^{k+1})-F_-(\varphi^{k}) &\leq
  F'_-(\varphi^{k})(\varphi_h^{k+1}-\varphi^k).\label{eq:FD:Fminus_ineq}
\end{align}

%%%%%
The inequality is now obtained from testing \eqref{eq:TD:chns1_solenoidal} with $v_h^{k+1}$,
\eqref{eq:TD:chns2_solenoidal} with $\mu_h^{k+1}$,
\eqref{eq:TD:chns3_solenoidal} with $(\varphi_h^{k+1} -\mathcal{P}^{k+1}\varphi^k)/\tau$,
and adding the resulting equations. This leads to
\begin{align*}
  % time derivative
  \frac{1}{2\tau}(\rho^kv_h^{k+1}-\rho^{k-1}v^k,v_h^{k+1})
  +\frac{1}{2\tau}(\rho^{k-1}(v_h^{k+1}-v^k),v_h^{k+1}) \\
  %transport
  +a(\rho^k v + J^k,v_h^{k+1},v_h^{k+1})
  %diffusion
  +(2\eta^k Dv_h^{k+1}:Dv_h^{k+1})
  %interfacial force
  -(\mu_h^{k+1}\nabla \varphi^k,v_h^{k+1})\\
  %transport
  +\frac{1}{\tau}(\varphi_h^{k+1}-\mathcal{P}^{k+1}\varphi^k,\mu_h^{k+1})
  + (v_h^{k+1}\nabla\varphi^k,\mu_h^{k+1})
  +(m^k\nabla \mu_h^{k+1},\nabla \mu_h^{k+1})\\
  %chemical potential
  +\sigma\epsilon\frac{1}{\tau}(\nabla
  \varphi_h^{k+1},\nabla(\varphi_h^{k+1}-\mathcal{P}^{k+1}\varphi^{k}))
  -\frac{1}{\tau}(\mu_h^{k+1},\varphi_h^{k+1}-\mathcal{P}^{k+1}\varphi^k)\\
  %free energy
  + \frac{1}{\tau}(F'_+(\varphi_h^{k+1}),\varphi_h^{k+1}-\mathcal{P}^{k+1}\varphi^k)
  + \frac{1}{\tau}(F'_-(\varphi^k),\varphi_h^{k+1}-\mathcal{P}^{k+1}\varphi^k)\\
  %gravity
  -(\rho^k g,v^{k+1}_h)
  &=0.
\end{align*}
The equalities  \eqref{eq:FD:rho_v_ineq} and \eqref{eq:FD:varphi_ineq} and the inequalities
\eqref{eq:FD:Fplus_ineq} and \eqref{eq:FD:Fminus_ineq} now imply
\begin{align*}
  \frac{1}{2\tau}\int_\Omega \left(\rho^k |v_h^{k+1}|^2 +
  \rho^{k-1}|v_h^{k+1}-v^k|^2 -\rho^{k-1}|v^k|^2
  \right)\,dx \\
  +\int_\Omega 2\eta^k |Dv_h^{k+1}|^2 \, dx
  +\int_\Omega m^k|\nabla \mu_h^{k+1}|^2\,dx \\
  + \frac{\sigma\epsilon}{2\tau}\int_\Omega |\nabla \varphi_h^{k+1}|^2\,dx
  -\frac{\sigma\epsilon}{2\tau} \int_\Omega |\nabla \mathcal{P}^{k+1}\varphi^{k}|^2 \,dx
  + \frac{\sigma\epsilon}{2\tau}\int_\Omega |\nabla \varphi_h^{k+1} - \nabla
  \mathcal{P}^{k+1}\varphi^k|^2\,dx\\
  +  \frac{1}{\tau}\int_\Omega \left(F(\varphi_h^{k+1})-F(\mathcal{P}^{k+1}\varphi^k)\right)\,dx
   -\int_\Omega\rho^k gv^{k+1}_h\,dx
  &\leq 0,
\end{align*}
which is the claim.
\end{proof}

\begin{theorem}\label{thm:FD:uniqueSolution}
System \eqref{eq:FD:chns1_solenoidal}--\eqref{eq:FD:chns3_solenoidal} admits a unique solution.
\end{theorem}
\begin{proof}
Assume there exist two different solutions to
\eqref{eq:FD:chns1_solenoidal}--\eqref{eq:FD:chns3_solenoidal} denoted by $(v^1,\varphi^1,\mu^1)$ and
$(v^2,\varphi^2,\mu^2)$. We show that the difference $v = v^1-v^2,\varphi = \varphi^1-\varphi^2,\mu
= \mu^1-\mu^2$ is zero.

After inserting the two solutions into \eqref{eq:FD:chns1_solenoidal}--\eqref{eq:FD:chns3_solenoidal}
and substracting the two sets of equations we perform the same steps as for the derivation of the
discrete energy estimate, Theorem \ref{thm:FD:energyInequality}, and obtain
\begin{align*}
  0 =& \frac{1}{2}\int_\Omega (\rho^k+\rho^{k-1})v^2\,dx + 2\tau\int_\Omega \eta^k|Dv|^2\,dx \\
  &+\tau \|\sqrt{m^k}\nabla \mu\|^2 + \sigma\epsilon\|\nabla \varphi\|^2
  +\left( F'_+(\varphi^1)-F_+'(\varphi^2),\varphi^1-\varphi^2 \right).
\end{align*}
Since all these terms are non negative we obtain
\begin{align*}
  \frac{1}{2}\int_\Omega (\rho^k+\rho^{k-1})v^2\,dx=&0,&
  \int_\Omega \eta^k|Dv|^2\,dx =&0,\\
  \|\nabla \mu\|^2=&0,&
  \|\nabla \varphi\|^2 =&0.
\end{align*}
Since both $\eta(\cdot)$ and $\rho(\cdot)$ are strictly positive by Assumption \ref{ass:rhoetamob_bounded}
we conclude $\|v\|_{H^1(\Omega)^n}=0$ and thus the uniqueness of the velocity field.

By testing \eqref{eq:FD:chns2_solenoidal} by $\Phi\equiv 1$
we obtain $(\varphi^1,1) = (\varphi^2,1) = (\mathcal{P}^{k+1}\varphi^k,1)$ and thus
$(\varphi^1-\varphi^2,1)=0$. Poincar\'e-Friedrichs inequality then yields
$\|\varphi\|_{H^1(\Omega)}=0$, and thus the uniqueness of the phase field.

Last we directly obtain that the chemical potential is unique up to a constant. By testing
\eqref{eq:FD:chns3_solenoidal} with $\Psi\equiv 1$ and inserting the two solutions we obtain
$(\mu^1-\mu^2,1) = (F_+'(\varphi^1)-F_+'(\varphi^2),1) = 0$ and thus $\|\mu\|_{H^1(\Omega)}=0$,
again by using Poincar\'e-Friedrichs inequality.
\end{proof}

Theorem \ref{thm:FD:energyInequality} estimates the Ginzburg Landau energy of the current phase
field $\varphi^{k+1}$ against the Ginzburg Landau energy of the projection of the old
phase field $\mathcal P^{k+1}(\varphi^k)$. Our aim is to obtain global in time inequalities
estimating the energy of the new phase field against the energy of the old phase field at each time
step. For this purpose let us state an assumption that later will be justified.

\begin{assumption}\label{ass:FP_leq_F}
Let $\varphi^k\in \mathcal{V}^1(\mathcal{T}^{k})$ denote the phase field at time instance $t_k$.
Let $\mathcal P^{k+1}\varphi^k\in \mathcal{V}^1(\mathcal{T}^{k+1})$ denote the projection of
$\varphi^k$ in $\mathcal{V}^1(\mathcal{T}^{k+1})$. We assume that there holds
\begin{align}
  F(\mathcal P^{k+1}\varphi^k)
  + \frac{1}{2}\sigma\epsilon |\nabla \mathcal P^{k+1}\varphi^{k}|^2
  \leq
  F(\varphi^k)
  +\frac{1}{2}\sigma\epsilon |\nabla \varphi^{k}|^2. \label{eq:FP_leq_F}
\end{align}
\end{assumption}
This assumption  means, that the Ginzburg Landau energy is not increasing through projection. Thus
no energy is numerically produced. 

Assumption \ref{ass:FP_leq_F} is in general not fulfilled for arbitrary sequences
$(\mathcal{T}^{k+1})$ of triangulations.
To ensure \eqref{eq:FP_leq_F} we add a post processing step to the adaptive space meshing, see
Section \ref{sec:Adaptivity}.

\begin{theorem}
Assume that for every $k=0,1,\ldots$ Assumption \ref{ass:FP_leq_F} holds.
Then for every $1\leq k < l$ we have
\begin{align*}
  \frac{1}{2}(\rho^{k-1}_hv^k_h,v^k_h) +& \int_\Omega F(\varphi^k_h)\,dx +
  \frac{1}{2}\sigma\epsilon (\nabla \varphi^{k}_h,\nabla \varphi^k_h)
  + \tau\sum_{m=k}^{l-1}(\rho^{m}g,v_h^{m+1})\\
  &\geq
  \frac{1}{2}(\rho^{l-1}v^{l}_h,v^{l}_h) + \int_\Omega F(\varphi^{l}_h)\,dx +
  \frac{1}{2}\sigma\epsilon (\nabla \varphi^{l}_h,\nabla \varphi^{l}_h)\\
  & + \sum_{m=k}^{l-1} (\rho^{m-1}(v^{m+1}_h-v^m_h),(v^{m+1}_h-v^m_h))\\
  & +\tau\sum_{m=k}^{l-1} (2\eta^{m}Dv^{m+1}_h,Dv^{m+1}_h)\\
  & +\tau\sum_{m=k}^{l-1} (m(\varphi^m_h)\nabla \mu ^{m+1}_h,\nabla \mu^{m+1}_h)\\
  & +\frac{1}{2}\sigma\epsilon 
  \sum_{m=k}^{l-1}( \nabla \varphi^{m+1}_h-\nabla \mathcal{P}^{m+1}\varphi^{m}_h,
  \nabla \varphi^{m+1}_h-\nabla \mathcal{P}^{m+1}\varphi^{m}_h).
\end{align*}
\end{theorem}
\begin{proof}
The stated result is obtained immediately from the energy estimate over one time step
\eqref{thm:FD:energyInequality} together with the Assumption \ref{ass:FP_leq_F}.
\end{proof}

\begin{remark} \label{rm:ProlongationMassDeviation}
We note that using $\Phi=1$ in \eqref{eq:FD:chns2_solenoidal} and using integration by parts only
delivers $(\varphi^{k+1}_h,1) = (\mathcal{P}^{k+1}\varphi^k,1)$ instead of $(\varphi^{k+1}_h,1) =(\varphi^k,1)$. If  we use
the quasi interpolation operator $Q^{k+1}$ introduced by Carstensen in \cite{Carstensen_QuasiInterpolation}
for our generic projection $\mathcal{P}^{k+1}$, we would obtain $(\varphi^{k+1}_h,1)
=(\varphi^k,1)$ since $Q^{k+1}$ preserves the mean value, i.e.
$(\varphi^{k+1}_h,1) =(Q^{k+1}\varphi^k,1)$ $\forall \varphi \in L^1(\Omega)$.

On the other hand if we use Lagrange interpolation $I^{k+1}$ we have
$|(I^{k+1}\varphi^k,1)_T - (\varphi^k,1)_T|  \leq Ch_T^3 \|\varphi^k\|_T$, and the deviation of
$(I^{k+1}\varphi^k,1)$ from $(\varphi^k,1)$ remaines small if we use bisection as refinement
strategy, since then $I^{k+1}\varphi^k \in \mathcal{V}^1(\mathcal{T}^{k+1})$ and
$\varphi^k\in\mathcal{V}^1(\mathcal{T}^{k})$ only differ on coarsened patches.

\end{remark}

\subsection{Existence of a solution to the time discrete system}\label{ssec:ExSol_TD}
Now we have shown that there exists a unique solution to
\eqref{eq:FD:chns1_solenoidal}--\eqref{eq:FD:chns3_solenoidal}. The energy
inequality can be used to obtain uniform bounds on the solution and will be used to obtain a
solution to the time discrete system \eqref{eq:TD:chns1_solenoidal}--\eqref{eq:TD:chns3_solenoidal} by a Galerkin
method.

\begin{theorem}\label{thm:TD:exSol}
Let $v^{k}\in  H(\mbox{div},\Omega)$,  $\varphi^{k-1}\in
H^1(\Omega)$, $\varphi^k\in H^1(\Omega)$, and $\mu^k\in W^{1,q}(\Omega),q>n$ be given data.
Then there exists a weak solution to \eqref{eq:TD:chns1_solenoidal}--\eqref{eq:TD:chns3_solenoidal}.
Moreover,
$\varphi^{k+1}\in H^2(\Omega)$ and $\mu^{k+1}\in H^2(\Omega)$ holds.
\end{theorem}
\begin{proof}
We proceed as follows.
We construct a sequence of meshes $(\mathcal{T}^{k+1}_{l})_{l\to \infty}$ with gridsize $h_l
\stackrel{l\to \infty}{\longrightarrow} 0$. We show that the sequence
$(v^{k+1}_{l},\varphi^{k+1}_{l},\mu^{k+1}_{l})$ of  unique and discrete solutions to
\eqref{eq:FD:chns1_solenoidal}--\eqref{eq:FD:chns3_solenoidal} is bounded
independently of $l$, and thus a weakly convergent subsequence exists which we show to converge to a
weak solution of \eqref{eq:TD:chns1_solenoidal}--\eqref{eq:TD:chns3_solenoidal}.
\bigskip

Let us start with defining the sequence of meshes. Let $\mathcal{T}^{k+1}_0 = \mathcal{T}^{k+1}$
and $\mathcal{T}^{k+1}_{l+1}$, $l=0,1,\ldots$, be obtained from $\mathcal{T}^{k+1}_{l}$ by bisection of
all triangles. The projection onto $\mathcal{T}^{k+1}_l$ we denote by
$\mathcal{P}_l^{k+1}$.

From the discrete energy inequality \eqref{eq:FD:energyInequality} we obtain
\begin{align*}
  \frac{1}{2}\int_\Omega \rho^k\left|v^{k+1}_l\right|^2\,dx
  + \frac{\sigma\epsilon}{2}\int_\Omega |\nabla \varphi^{k+1}_l|^2\,dx
  +\int_\Omega F(\varphi^{k+1}_l)\,dx\\
  + \frac{1}{2}\int_\Omega \rho^{k-1}|v^{k+1}_l-v^k|^2\,dx
  + \frac{\sigma\epsilon}{2}
  \int_\Omega |\nabla \varphi^{k+1}_l-\nabla \mathcal{P}_l^{k+1} \varphi^k|^2\,dx\\
  + \tau\int_\Omega 2\eta^k|Dv^{k+1}_l|^2\,dx
  +\tau\int_\Omega m^k|\nabla \mu^{k+1}_l|^2\,dx \\
  \leq
  \frac{1}{2}\int_\Omega \rho^{k-1}\left|v^{k}\right|^2\,dx
  + \frac{\sigma\epsilon}{2}\int_\Omega |\nabla \mathcal{P}_l^{k+1}\varphi^{k}|^2\,dx
  +\int_\Omega F(\mathcal{P}_l^{k+1}\varphi^{k})\,dx + \tau\int_\Omega \rho^k g v^{k+1}_{l} .
\end{align*}
We have the stability of the projection operator and thus
\begin{align*}
  \int_\Omega |\nabla \mathcal{P}_l^{k+1}\varphi^{k}|^2\,dx
  \leq \|\nabla \varphi^{k}\|_{L^2(\Omega)}^2.
\end{align*}
Due to  Assumption \ref{ass:F_boundedPoly} on $F$  there exists  a
constant $C>0$ such that
\begin{align*}
  \int_\Omega F(\mathcal{P}_l^{k+1}\varphi^k)\,dx
  \leq&  C\int_\Omega |  \mathcal{P}_l^{k+1}\varphi^k|^q + 1 \, dx\\
  \leq & C \left( \|\mathcal{P}_l^{k+1}\varphi^k\|_{L^q(\Omega)}^q +1 \right)\\
  \leq& C  \left(\|\varphi^{k}\|_{L^q(\Omega)}^q + 1\right),
\end{align*}
where we again use the $L^q$-stability of the projection operator  together with the Sobolev
embedding $H^1(\Omega) \hookrightarrow L^q(\Omega)$ with $q$ as in Assumption \ref{ass:F_boundedPoly}.
By using H\"older's inequality and Young's inequality we further have
\begin{align*}
  \tau\int_\Omega \rho^k g v^{k+1}_l\, dx 
  \leq &\tau \left(\int_\Omega \rho^k|g|^2\,dx\right)^{1/2} 
  \left(\int_\Omega \rho^k |v^{k+1}_l|^2\, dx\right)^{1/2}\\
  \leq & \tau^2\int_\Omega \rho^k|g|^2\,dx + \frac{1}{4}\int_\Omega \rho^k |v^{k+1}_l|^2\, dx
\end{align*}

Since $\rho^{k-1}>0$, $\rho^k>0$, $\eta^k>0$, and  $m^k>0$ by Assumption \ref{ass:rhoetamob_bounded}
we obtain that
$\|v^{k+1}_l\|_{H^1(\Omega)}$, $\|\nabla \varphi^{k+1}_l\|$ and $\|\nabla \mu^{k+1}_l\|$ are
uniformly bounded independend of $l$.

By inserting $\Phi \equiv 1$ in \eqref{eq:FD:chns2_solenoidal} we obtain
$(\mathcal{P}_l^{k+1}\varphi^k,1) =(\varphi^{k+1}_l,1)$
and by Poincar\'e-Friedrichs inequality thus
\begin{align*}
  \|\varphi^{k+1}_l\|_{H^1(\Omega)} &\leq
  C\left( \|\nabla \varphi^{k+1}_l\| +(\mathcal P_l^{k+1}\varphi^k,1) \right)\\
  &\leq C \left( \|\nabla \varphi_l^{k+1}\| + \| \mathcal{P}_l^{k+1}\varphi^k\| \right)\\
  &\leq C\left( \|\nabla \varphi_l^{k+1}\| + \|\varphi^k\|  \right).
\end{align*}
Thus $\| \varphi^{k+1}_l\|_{H^1(\Omega)}$  is uniformly bounded.

By inserting $\Psi \equiv 1$ in \eqref{eq:FD:chns3_solenoidal} we obtain
$(\mu^{k+1}_l,1) = (F'_+(\varphi^{k+1}_l)+ F'_-(\mathcal{P}_l^{k+1}\varphi^k),1)$.
Due to the Assumption \ref{ass:F_boundedPoly} on $F'_+$  the first part can be bounded by
$C(\|\varphi^{k+1}_l\|_{L^q(\Omega)}^q+1)$ which is bounded by Sobolev embedding.
Also due to  Assumption \ref{ass:F_boundedPoly} on $F_-$  and due to the $L^q$ stability of
$\mathcal P_l^{k+1}$ the second part can be bounded by $C(\|\varphi^{k}\|_{L^q(\Omega)}^q+1)$.
Thus, by the same arguments as $\|\varphi^{k+1}_l\|_{H^1(\Omega)}$, also
$\|\mu^{k+1}_l\|_{H^1(\Omega)}$ is uniformly bounded.

Consequently there exist $\overline  v\in H^1_0(\Omega)^n, \overline  \varphi\in H^1(\Omega),
\overline \mu\in H^1(\Omega)$ and a subsequence $l_i$ such that
$v^{k+1}_{l_i} \rightharpoonup \overline  v$ in $H^1_0(\Omega)^n$,
$\varphi^{k+1}_{l_i} \rightharpoonup \overline  \varphi$ in $H^1(\Omega)$,
$\mu^{k+1}_{l_i} \rightharpoonup \overline  \mu$  in $H^1(\Omega)$
for $l_i\to \infty$.
\bigskip

We show that this triple of functions indeed is a weak solution to
\eqref{eq:TD:chns1_solenoidal}--\eqref{eq:TD:chns3_solenoidal}.
Inserting the sequence into \eqref{eq:TD:chns1_solenoidal}--\eqref{eq:TD:chns3_solenoidal} yields
\begin{align}
  \frac{1}{2\tau}\int_\Omega \left( \rho^kv^{k+1}_{l_i}-\rho^{k-1}v^{k} \right) w
  +  \rho^{k-1}(v^{k+1}_{l_i}-v^k)w\,dx\nonumber\\
  +a(\rho^kv^k + J^k,v^{k+1}_{l_i},w)
  +\int_{\Omega} 2\eta^k Dv^{k+1}_{l_i}:Dw\,dx\nonumber \\
  -\int_\Omega \mu^{k+1}_{l_i}\nabla \varphi^k
  w + \rho^kgw\,dx &= 0,  \, \forall w \in H(\mbox{div},\Omega)\\
  \tau^{-1}\int_\Omega (\varphi^{k+1}_{l_i}-\varphi^k) \Phi \,dx +
  \int_\Omega(v^{k+1}_{l_i}\cdot \nabla \varphi^k) \Phi \, dx \nonumber\\
  + \int_\Omega m(\varphi^k)\nabla \mu^{k+1}_{l_i}\cdot\nabla \Phi\,dx &=0, \,\forall
  \Phi\in H^1(\Omega)\\
  \sigma \epsilon\int_\Omega\nabla \varphi^{k+1}_{l_i}\cdot\nabla \Psi\,dx
  - \int_\Omega \mu^{k+1}_{l_i}\Psi\,dx\nonumber\\
  +\int_\Omega ((F_+)'(\varphi^{k+1}_{l_i})+(F_-)'(\varphi^k))\Psi\,dx &= 0. \,\forall
  \Psi\in H^1(\Omega).
\end{align}
Now there holds
\begin{align*}
  \frac{1}{2\tau}\int_\Omega  \left(\rho^k + \rho^{k-1}\right) v^{k+1}_{l_i} w\,dx
  &\leq \frac{1}{\tau}\overline \rho \|v_{l_i}^{k+1}\|_{L^2(\Omega)}\|w\|_{L^2(\Omega)}
\end{align*}
and thus
$ \frac{1}{2\tau}\int_\Omega\left(\rho^k + \rho^{k-1}\right)  w \cdot \,dx \in
(H^1_0(\Omega)^n)^*$
yielding
\begin{align*}
  \frac{1}{2\tau}\int_\Omega\left(\rho^k + \rho^{k-1}\right) v^{k+1}_{l_i} w\,dx \to
  \frac{1}{2\tau}\int_\Omega\left(\rho^k + \rho^{k-1}\right) \overline v w\,dx.
\end{align*}

Since $\mu^k\in W^{1,q}(\Omega), q>n$ there holds $J^k\in L^q(\Omega)$ and thus by Sobolev
embedding  we obtain
\begin{align*}
  \left|\int_\Omega \left(\left( \left(\rho^kv^k + J^k\right)\cdot \nabla
  \right)v^{k+1}_{l_i}\right) w\,dx\right|
  \leq& C \|\left(\rho^kv^k + J^k\right)w\|\|\nabla v^{k+1}_{l_i}\|,\\
  \left|\int_\Omega \left(\left( \left(\rho^kv^k + J^k\right)\cdot \nabla
  \right)w\right) v^{k+1}_{l_i}\,dx\right|
  \leq& C \|\left(\left(\rho^kv^k + J^k\right)\nabla\right) w\|_{L^{\frac{2q}{q+2}}}
  \|v^{k+1}_{l_i}\|_{L^{\frac{2q}{q-2}}},
\end{align*}
and thus
$a(\rho^kv^k + J^k,\cdot,w) \in (H^1_0(\Omega)^n)^*$.
This gives
\begin{align*}
  a(\rho^kv^k + J^k,v^{k+1}_{l_i},w) \to a(\rho^kv^k + J^k,\overline v,w)
\end{align*}
The convergence of the remaining terms can be concluded in a similar manner.

Since $\varphi^{k+1}_{l_i} \rightharpoonup \overline \varphi$ in $H^1(\Omega)$ there exists
a subsequence, again denoted by $l_i$ such that $\varphi^{k+1}_{l_i}\to \overline \varphi$
in $L^q(\Omega)$, $q$ as in Assumption \ref{ass:F_boundedPoly}. From Assumption \ref{ass:F_boundedPoly} and the
dominated convergence theorem we thus obtain
\begin{align*}
  \int_\Omega F_+'(\varphi^{k+1}_{l_i})\Psi\,dx \to   \int_\Omega F_+'(\overline \varphi)\Psi\,dx.
\end{align*}

Next we show the weak solenoidality of $\overline v$.
To begin with we note that every $q\in L^2_{(0)}(\Omega)$ can be approximated by
a sequence
$(q_l)_{l\in \mathbb{N}} \subset \mathcal{V}^1(\mathcal T^{k+1}_l)\cap L^2_{(0)}(\Omega)$, so that
for every $\xi > 0$  an index $N_\xi$ exists, such that $\|q-q_l\| \leq \xi$ for $l\geq N_\xi$.
Now we have for arbitrary $q\in L^2_{(0)}(\Omega)$
\begin{align*}
  |(\mbox{div }\overline v,q)|
  &\leq | (\mbox{div }\overline v,q-q_l) |
  + | (\mbox{div } \overline v - \mbox{div }v_{l_i},q_l)|
  + |(\mbox{div }v_{l_i},q_l)|.
\end{align*}
Let $\xi>0$ be given.
For the first addend we have
$| (\mbox{div }\overline v,q-q_l) | \leq \|\mbox{div }\overline v\| \|q-q_l\| \leq C\xi$ for
$l\geq N_\xi$.

Since the sequence $q_l$ is defined on the same hierarchy of meshes as $v_{l}$ we may restrict $q_l$
to the subsequence $l_i$ and obtain that both $q_{l_i}$ and $v_{l_i}$ are defined on the same
meshes.
We set $k := \min\{l_i\,|\,l_i\geq N_\xi\}$.
Now  we have $(\mbox{div }v_{l_i},q_{k})=0$ for $l_i\geq k$, since then $q_k\in
\mathcal{V}^1(\mathcal{T}^{k+1}_{l_i})$,
i.e. the third addend vanishes.
By choosing $l_i$ so large that
$| (\mbox{div } \overline v - \mbox{div }v_{l_i},q_{k})| \leq C\xi$ holds by weak convergence of
$v_{l_i}$, the weak solenoidality of $\overline v$ is shown, since $\xi>0$ is chosen arbitrarily.

Thus the triple $\overline v,\overline \varphi,\overline \mu$ indeed is a weak solution.

\bigskip

It remains to obtain the stated higher regularity for  $\mu^{k+1}$ and $\varphi^{k+1}$.
This directly follows by regularity results for the Laplacian, see \cite[Thm. 3.10]{ErnGuermond}.
Since $\mu^{k+1}-F'_+(\varphi^{k+1})- F'_-(\varphi^k) \in L^2(\Omega)$
it follows that $\varphi^{k+1} \in H^2(\Omega)$ and thus, since
we have $\tau^{-1}(\varphi^{k+1}-\varphi^k)+v^{k+1}\nabla \varphi^k \in L^2(\Omega)$,
also $\mu^{k+1}\in H^2(\Omega)$.
\end{proof}

The uniqueness of the solution follows by the same steps as the uniqueness of the discrete
solutions, see Theorem \ref{thm:FD:uniqueSolution}.
Like the fully discrete scheme, also the
time-discrete scheme fulfills an energy inequality.

\begin{theorem}\label{thm:TD:energyInequality}
Let $\varphi^{k+1},\mu^{k+1},v^{k+1}$ be a solution to
\eqref{eq:TD:chns1_solenoidal}--\eqref{eq:TD:chns3_solenoidal}. Then the following energy inequality holds.
\begin{align*}
  \frac{1}{2}\int_\Omega \rho^k\left|v^{k+1}\right|^2\,dx
  + \frac{\sigma\epsilon}{2}\int_\Omega |\nabla \varphi^{k+1}|^2\,dx
  +\int_\Omega F(\varphi^{k+1})\,dx\nonumber\\
  + \frac{1}{2}\int_\Omega \rho^{k-1}|v^{k+1}-v^k|^2\,dx
  + \frac{\sigma\epsilon}{2}\int_\Omega |\nabla \varphi^{k+1}-\nabla \varphi^k|^2\,dx\nonumber\\
  + \tau\int_\Omega 2\eta^k|Dv^{k+1}|^2\,dx
  +\tau\int_\Omega m^k|\nabla \mu^{k+1}|^2\,dx \nonumber\\
  \leq
  \frac{1}{2}\int_\Omega \rho^{k-1}\left|v^{k}\right|^2\,dx
  + \frac{\sigma\epsilon}{2}\int_\Omega |\nabla \varphi^{k}|^2\,dx
  +\int_\Omega F(\varphi^{k})\,dx. + \int_\Omega \rho^kg v^{k+1}
\end{align*}
\end{theorem}
\begin{proof}

The inequality is obtaind from testing \eqref{eq:TD:chns1_solenoidal} with $v^{k+1}$,
\eqref{eq:TD:chns2_solenoidal} with $\mu^{k+1}$,
\eqref{eq:TD:chns3_solenoidal} with $(\varphi^{k+1} -\varphi^k)/\tau$ and using the same arguments
as in the proof for Theorem \ref{thm:FD:energyInequality}.
\end{proof}

\begin{remark}
Let $F$ denote the relaxed double-obstacle free energy introduced in Remark~\ref{rm:freeEnergies},
with relaxation parameter $\hp$. Let $(v_\hp,\varphi_\hp,\mu_\hp)_{\hp\in  \mathbb{R}}$ denote the
sequence of solutions of \eqref{eq:TD:chns1_solenoidal}--\eqref{eq:TD:chns3_solenoidal} for a
sequence $(\hp_l)_{l\in\mathbb{N}}$. 
From the linearity of \eqref{eq:TD:chns1_solenoidal} and
\cite[Prop. 4.2]{HintermuellerHinzeTber} it follows, that there exists a subsequence, still denoted
by $(v_\hp,\varphi_\hp,\mu_\hp)_{\hp\in  \mathbb{R}}$, such that
\begin{align*}
  (v_\hp,\varphi_\hp,\mu_\hp)_{\hp\in  \mathbb{R}} 
  \to (v^*,\varphi^*,\mu^*)\quad \mbox{ in }  H^1(\Omega),
\end{align*}
where $(v^*,\varphi^*,\mu^*)$ denotes the solution of 
\eqref{eq:TD:chns1_solenoidal}--\eqref{eq:TD:chns3_solenoidal},
where $F^{obst}$, denoted in Remark~\ref{rm:freeEnergies}, is chosen as free energy. Especially
$|\varphi^*|\leq 1$ holds. In the following argumentation we concentrate on the phase field only.
From the regularity $\varphi_\hp\in H^2(\Omega)$ 
together with a-priori estimates on the solution of the Poisson problem
and the energy inequality of Theorem \ref{thm:TD:energyInequality}, 
we obtain the existence of a strongly convergent subsequence 
$\varphi_{\hp'}\to \varphi^*$ in $C^{0,\alpha}(\overline \Omega)$, where we use the compact
embedding $H^2(\Omega) \hookrightarrow C^{0,\alpha}(\overline \Omega)$ for $2\alpha < 4-n$.

Thus for $\hp$ large enough we have $|\varphi_\hp|\leq 1 +\theta$ with $\theta$ arbitrarily small. 
Currently we are not able to quantify how large $\hp$ has to be chosen in dependence of $\theta$
to guarantee this bound.
Therefore we use the cut-off procedure described before Remark \ref{rm:freeEnergies}.
\end{remark}

%%%%%%%%%%%%%%%%%%%%%%%%%%%%%%%%%%%%%%%%%%%%%%%%%%%%%%%%%%%%%%%%%%%%%%%%%%%%%%%%%%%
%%%%%%%%%%%%%%%%%%%%%%%%%%%%%%%%%%%%%%%%%%%%%%%%%%%%%%%%%%%%%%%%%%%%%%%%%%%%%%%%%%%
%%%%%%%%%%%%%%%%%%%%%%%%%%%%%%%%%%%%%%%%%%%%%%%%%%%%%%%%%%%%%%%%%%%%%%%%%%%%%%%%%%%
%%%%%%%%%%%%%%%%%%%%%%%%%%%%%%%%%%%%%%%%%%%%%%%%%%%%%%%%%%%%%%%%%%%%%%%%%%%%%%%%%%%
%Here we state the adaptive concept and the adaptivity cycle together with the postprocessing
%labels with AD
\section{The A-Posteriori Error Estimation}\label{sec:Adaptivity}
For an efficient solution of \eqref{eq:FD:chns1_solenoidal}--\eqref{eq:FD:chns3_solenoidal} we next describe
an a-posteriori error estimator based mesh refinement scheme that is reliable and efficient up to terms
of higher order and errors introduced by the projection. We also describe how
Assumption \ref{ass:FP_leq_F} on the evolution  of the free energy, given  in
\eqref{eq:FD:energyInequality}, under projection  is fulfilled in the discrete setting.

Let us briefly comment on available adaptative concepts for the spatial discretization of
Cahn--Hilliard Navier--Stokes systems. Heuristic approaches exploiting knowledge of the location
of the diffuse interface can be found in
\cite{KayStylesWelford,blank_solvingCHwithNewton,Aland_Voigt_bubble_benchmark,Gruen_Klingbeil_CHNS_AGG_numeric}.
In \cite{HintermuellerHinzeKahle_adaptiveCHNS} a fully adaptive, reliable and efficient, residual
based error estimator for the Cahn--Hilliard part in the Cahn--Hilliard Navier--Stokes system is
proposed, which extends the results of \cite{HintermuellerHinzeTber} for Cahn--Hilliard to
Cahn--Hilliard Navier--Stokes systems with Moreau--Yosida relaxation of the double-obstacle free
energy.
A residual based error estimator for Cahn--Hilliard systems with double-obstacle free energy is
proposed in  \cite{BanasNurnbergAPosteriori}.

In the present section we propose an all-in-one adaptation concept for the fully coupled
Cahn--Hilliard Navier--Stokes system, where we exploit the energy inequality of Theorem
\ref{thm:FD:energyInequality}. To the best of the author's knowledge this is the first contribution
to the fully adaptive treatment of the fully coupled Cahn--Hilliard Navier--Stokes
system.

\subsection*{The fully discrete system used in the numerical realization}
Since in our numerical realization we do not include the solenoidality of the velocity field $v$
into the discrete Ansatz space we now introduce a  weak formulation for the time discrete version
of \eqref{eq:CHNS1_weak}--\eqref{eq:CHNS3_weak} in primitive variables, which by
\cite{GiraultRaviart_FEM_for_NavierStokes} is equivalent to
\eqref{eq:TD:chns1_solenoidal}--\eqref{eq:TD:chns3_solenoidal}:\\
For $k\geq 1$, given $\varphi^{k-1}\in H^1(\Omega)$, $\varphi^k\in H^1(\Omega)$, $\mu^k \in
W^{1,q}(\Omega),q>n$, $v^k \in H^1_0(\Omega)^n$ find
$v^{k+1} \in H^1_0(\Omega)^n$, $p^{k+1}\in L^2_{(0)}(\Omega)$, $\varphi^{k+1}\in H^1(\Omega)$, and
$\mu^{k+1}\in H^1(\Omega)$ satisfying
\begin{align}
  \frac{1}{2\tau}( \rho^kv^{k+1}-\rho^{k-1}v^{k} +  \rho^{k-1}(v^{k+1}-v^k),w)
  &\nonumber\\
  +a(\rho^k v^k+J^k,v^{k+1},w)+(2\eta Dv^{k+1}:Dw)&\nonumber\\
  -(p^{k+1},\mbox{div}(w)) - (\mu^{k+1}\nabla \varphi^k + \rho^k g, w) &= 0
  &&\forall w \in H^1_0(\Omega)^n,\label{eq:AD:chns1_timeDisc}\\
  -(\mbox{div}(v^{k+1}),q)&= 0 &&\forall q\in L^2_{(0)}(\Omega),\label{eq:AD:chns2_timeDisc}\\
  \frac{1}{\tau}(\varphi^{k+1}-\varphi^k, \Phi) +
  (v^{k+1}\cdot \nabla\varphi^k, \Phi)
  +  (m(\varphi^k)\nabla \mu^{k+1},\nabla \Phi) &=0 &&\forall
  \Phi\in H^1(\Omega),\label{eq:AD:chns3_timeDisc}\\
  \sigma \epsilon(\nabla \varphi^{k+1},\nabla \Psi)
  +((F_+)'(\varphi^{k+1})+(F_-)'(\varphi^k),\Psi)- (\mu^{k+1},\Psi) &= 0&&\forall
  \Psi\in H^1(\Omega).\label{eq:AD:chns4_timeDisc}
\end{align}
The corresponding fully discrete system now reads:\\
For $k\geq 1$, given
$\varphi^{k-1}\in H^1(\Omega)$,
$\varphi^k\in H^1(\Omega)$,
$\mu^k \in W^{1,q}(\Omega),q>n$,
$v^k \in H^1_0(\Omega)^n$
find
$v^{k+1}_h \in  \mathcal{V}^2(\mathcal{T}^{k+1})$,
$p_h^{k+1}\in \mathcal V^1(\mathcal T^{k+1}),\, \int_\Omega p_h^{k+1}\,dx= 0$,
$\varphi^{k+1}_h\in \mathcal V^1(\mathcal  T^{k+1})$,
$\mu^{k+1}_h \in \mathcal V^1(\mathcal T^{k+1})$
such that for all
$w \in \mathcal V^2(\mathcal T^{k+1})$,
$q \in \mathcal V^1(\mathcal T^{k+1})$,
$\Phi \in \mathcal V^1(\mathcal T^{k+1})$,
$\Psi \in \mathcal V^1(\mathcal T^{k+1})$
there holds:
\begin{align}
  \frac{1}{2\tau}(\rho^{k}v^{k+1}_h-\rho^{k-1}v^k+\rho^{k-1}(v^{k+1}_h-v^k),w)
  + a(\rho^kv^k+J^k,v_h^{k+1},w)\nonumber\\
  +(2\eta^kDv^{k+1}_h,\nabla w)-(\mu^{k+1}_h\nabla\varphi^{k} + \rho^k g,w) - (p_h^{k+1},\mbox{div}
  w) &= 0,\label{eq:AD:chns1_fullDisc}\\
  - (\mbox{div} v_h^{k+1},q) &= 0,\label{eq:AD:chns2_fullDisc}\\
  \frac{1}{\tau}(\varphi^{k+1}_h-\mathcal{P}^{k+1}\varphi^k,\Phi)+(m(\varphi^k)\nabla
  \mu^{k+1}_h,\nabla \Phi) -(v^{k+1}_h \varphi^k,\nabla \Phi)
  &=0,\label{eq:AD:chns3_fullDisc}\\
  \sigma\epsilon(\nabla \varphi^{k+1}_h,\nabla
  \Psi)+(F_+^\prime(\varphi^{k+1}_h)+F^\prime_-(\mathcal{P}^{k+1}\varphi^k),\Psi)-(\mu^{k+1}_h,\Psi)
  &=0. \label{eq:AD:chns4_fullDisc}
\end{align}

Thus we use the famous Taylor--Hood LBB-stable $P2-P1$ finite
element for the discretization of the velocity - pressure field and piecewise linear and continuous
finite elements for the discretization of the phase field and the chemical potential. For other kinds of possible
discretizations of the velocity-pressure field we refer to e.g.
\cite{Verfuerth_aPosteriori_new}.

Note that we perform integration by parts in \eqref{eq:AD:chns3_fullDisc} in the transport term. As
soon as $\mathcal{P}^{k+1}$ is a mass conservened projection we by testing equation
\eqref{eq:AD:chns3_fullDisc} with $\Phi = 1$ obtain the conservation of mass in the fully discrete scheme.

The link between equations
\eqref{eq:AD:chns1_fullDisc}--\eqref{eq:AD:chns4_fullDisc} and
\eqref{eq:FD:chns1_solenoidal}--\eqref{eq:FD:chns3_solenoidal} is provided by the next theorem.

\begin{theorem}
Let $v^{k+1}_h,\varphi^{k+1}_h,\mu^{k+1}_h$ denote the unique  solution to
\eqref{eq:FD:chns1_solenoidal}--\eqref{eq:FD:chns3_solenoidal}. Then there exists a unique pressure
$p_h^{k+1} \in \mathcal{V}^1(\mathcal{T}^{k+1}),\, \int_\Omega p_h^{k+1}\,dx= 0$ such that
$(v^{k+1}_h,p_h^{k+1},\varphi_h^{k+1},\mu_h^{k+1})$ is a solution to \eqref{eq:AD:chns1_fullDisc}--\eqref{eq:AD:chns4_fullDisc}.
 The opposite direction is obvious.
\end{theorem}
\begin{proof}
Since we use LBB-stable finite elements, from
\cite[Thm. II 1.1]{GiraultRaviart_FEM_for_NavierStokes} we obtain the stated result.
\end{proof}

\subsection*{Derivation of the error estimator}
We begin with noting that the special structure of our time discretization gives rise to an error
estimator which both estimates the error in the approximation of the velocity, and in the
approximation of the phase field and the chemical potential. We are not able to estimate the error in the approximation of the
pressure field and the estimator will only be reliable and efficient up to higher order terms.

In the derivation of  the estimator we follow \cite{HintermuellerHinzeTber} and restrict the
presentation of its construction to the main steps.

We define the following error terms:
\begin{align}
  e_v := &v^{k+1}_h - v^{k+1},&
  e_p := &p^{k+1}_h - p^{k+1},\label{eq:AD:err1}\\
  e_\varphi := &\varphi^{k+1}_h - \varphi^{k+1},&
  e_\mu := &\mu^{k+1}_h - \mu^{k+1},\label{eq:AD:err2}
\end{align}
as well as the discrete element residuals
\begin{align*}
  r^{(1)}_h :=&
  \frac{\rho^{k}+\rho^{k-1}}{2}v^{k+1}_h   - \rho^{k-1}v^k
  + \tau (b^k\nabla)v^{k+1}_h
  +\frac{1}{2}\tau \mbox{div}(b^k)v^{k+1}_h \\
  &-2\tau \mbox{div}\left(\eta^kDv^{k+1}_h\right)
  + \tau \nabla p^{k+1}_h
  - \tau \mu^{k+1}_h\nabla \varphi^k -  \rho^k g,\\
  r^{(2)}_h := & \varphi^{k+1}_h-\mathcal{P}^{k+1}\varphi^k + \tau v^{k+1}_h\nabla \varphi^k
  - \tau \mbox{div}(m^k\nabla \mu^{k+1}_h),\\
  r^{(3)}_h  := & F'_+(\varphi^{k+1}_h) + F'_-(\mathcal{P}^{k+1}\varphi^k) - \mu^{k+1}_h,
\end{align*}
where $b^k := \rho^kv^k+J^k $.
Furthermore we define the error indicators
\begin{equation}
  \begin{aligned}
\eta_T^{(1)} :=& h_T\|r^{(1)}_h\|_T, &
\eta_E^{(1)} :=& h_E^{1/2}\|2\eta^k \left[Dv^{k+1}_h\right]_\nu \|_E,\\
\eta_T^{(2)} := & h_T\|r^{(2)}_h\|_T, &
\eta_E^{(2)} := & h_E^{1/2}\|m^k\left[\nabla \mu^{k+1}_h\right]_\nu\|_E, \\
\eta_T^{(3)} := & h_T\|r^{(3)}_h\|_T, &
\eta_E^{(3)} := & h_E^{1/2}\|\left[\nabla \varphi^{k+1}_h\right]_\nu\|_E.
\end{aligned}
\label{eq:AD:etaT123_etaTE123}
\end{equation}
Here $\left[\cdot\right]_\nu$  denotes the jump of a discontinuous function in normal direction
$\nu$ pointing from the triangle with lower global number to the triangle with higher global
number. Thus $\eta_E^{(j)}$, $j=1,2,3$ measures  the jump of the corresponding variable
across the edge $E$, while $\eta_T^{(j)}$, $j=1,2,3$ measures the triangle wise residuals.

By $\Pi_h:H^1(\Omega)\to \mathcal{V}^j(\mathcal{T}^{k+1})$, $j=1,2$ we denote Cl\'ement's
interpolation operator \cite{Clement_Interpolation}, which satisfies for each triangle $T\in \mathcal{T}^{k+1}$ and each
edge $E\in \mathcal{E}^{k+1}$ the approximation estimates
\begin{align}
  \| v-\Pi_hv\|_T \leq & Ch_T\|\nabla v\|_{\omega_T} &&\forall v \in H^1(\Omega),
  \label{eq:AD:Clement1}\\
  \mbox{ and }\| v-\Pi_hv\|_E \leq & Ch_E^{1/2}\|\nabla v\|_{\omega_E} &&\forall v \in H^1(\Omega),
  \label{eq:AD:Clement2}
\end{align}
where $C$ is a generic positive constant and $\omega_T$ and $\omega_E$ are given by
\begin{align*}
  \omega_T := &
  \{ T' \in \mathcal{T}^{k+1}\, : \, \overline T \cap \overline{T'} \not = \emptyset\},\\
  \omega_E := &
  \{ T \in \mathcal{T}^{k+1}\, : \, E \subset \overline{T}\}.
\end{align*}
Subsequently it is clear whether $\Pi_h$ maps to $\mathcal{V}^1$
or to $\mathcal V^2$. We therefore do not introduce further supscripts to distinguish these two
cases.

In the following we write equations \eqref{eq:AD:chns1_timeDisc}--\eqref{eq:AD:chns4_timeDisc} as
\begin{align*}
  F^1( (v^{k+1},p^{k+1},\varphi^{k+1},\mu^{k+1}),w) &= 0,&
  F^2( (v^{k+1},p^{k+1},\varphi^{k+1},\mu^{k+1}),q) &= 0,\\
  F^3( (v^{k+1},p^{k+1},\varphi^{k+1},\mu^{k+1}), \Phi) &= 0,&
  F^4( (v^{k+1},p^{k+1},\varphi^{k+1},\mu^{k+1}), \Psi) &= 0,
\end{align*}
and analogously \eqref{eq:AD:chns1_fullDisc}--\eqref{eq:AD:chns4_fullDisc} as
\begin{align*}
  F^1_h( (v_h^{k+1},p_h^{k+1},\varphi_h^{k+1},\mu_h^{k+1}),w) &= 0,&
  F^2_h( (v_h^{k+1},p_h^{k+1},\varphi_h^{k+1},\mu_h^{k+1}),q) &= 0,\\
  F^3_h( (v_h^{k+1},p_h^{k+1},\varphi_h^{k+1},\mu_h^{k+1}), \Phi) &= 0,&
  F^4_h( (v_h^{k+1},p_h^{k+1},\varphi_h^{k+1},\mu_h^{k+1}), \Psi) &= 0.
\end{align*}

Since the error functions defined in \eqref{eq:AD:err1}--\eqref{eq:AD:err2} are valid test
functions for the system \eqref{eq:AD:chns1_timeDisc}--\eqref{eq:AD:chns4_timeDisc} we have
\begin{align*}
  F^1( (v^{k+1},p^{k+1},\varphi^{k+1},\mu^{k+1}), \tau e_v) &= 0,&
  F^2( (v^{k+1},p^{k+1},\varphi^{k+1},\mu^{k+1}), \tau e_p) &= 0,\\
  F^3( (v^{k+1},p^{k+1},\varphi^{k+1},\mu^{k+1}), \tau e_\mu) &= 0,&
  F^4( (v^{k+1},p^{k+1},\varphi^{k+1},\mu^{k+1}), e_\varphi) &= 0.
\end{align*}
Thus
\begin{align*}
  &F_h^1( (v^{k+1}_h,p^{k+1}_h,\varphi^{k+1}_h,\mu^{k+1}_h), \tau e_v)\\
  &=F_h^1( (v^{k+1}_h,p^{k+1}_h,\varphi^{k+1}_h,\mu^{k+1}_h), \tau e_v)
  - F^1( (v^{k+1},p^{k+1},\varphi^{k+1},\mu^{k+1}), \tau e_v).
\end{align*}
For $F_h^2,F_h^3,F_h^4$ similar expressions hold.\\
Summing up these expressions yields
\begin{align*}
  &\frac{1}{2}((\rho^k+\rho^{k-1})e_v,e_v) + 2\tau (\eta^kDe_v,De_v)\\
  &+\tau(m^k\nabla e_\mu,\nabla e_\mu) + \sigma\epsilon\|\nabla e_\varphi\|^2
  +(F'_+(\varphi^{k+1}_h)-F'_+(\varphi^{k+1}),e_\varphi)\\
  =&
  F_h^1( (v^{k+1}_h,p^{k+1}_h,\varphi^{k+1}_h,\mu^{k+1}_h), \tau e_v)
  \\
  &+F_h^2( (v^{k+1}_h,p^{k+1}_h,\varphi^{k+1}_h,\mu^{k+1}_h), \tau e_p)
  \\
  &+F_h^3( (v^{k+1}_h,p^{k+1}_h,\varphi^{k+1}_h,\mu^{k+1}_h), \tau e_\mu)
  \\
  &+F_h^4( (v^{k+1}_h,p^{k+1}_h,\varphi^{k+1}_h,\mu^{k+1}_h), \tau e_\varphi)
  \\
  &+(\mathcal{P}^{k+1}\varphi^k-\varphi^k,e_\mu) - (F'_-(\mathcal{P}^{k+1}\varphi^k) -
  F'_-(\varphi^k),e_\varphi)\\
  =& \eta_1 + \eta_2 + \eta_3 + \eta_4 + \eta_5+ \eta_6.
\end{align*}
Exemplarily we consider the term $\eta_4$.
Since $\Pi_he_\varphi \in \mathcal{V}^1(\mathcal{T}^{k+1})$  is a valid test
function for \eqref{eq:AD:chns1_fullDisc}--\eqref{eq:AD:chns4_fullDisc} we obtain
$$F_h^4( (v^{k+1}_h,p^{k+1}_h,\varphi^{k+1}_h,\mu^{k+1}_h), \Pi_he_\varphi) = 0.$$
Using \eqref{eq:AD:Clement1}, \eqref{eq:AD:Clement2} as well as
H\"older's and Cauchy-Schwarz's inequality, we have
\begin{align*}
  \eta_4 =&
  (r^{(3)}_h,e_\varphi-\Pi_he_\varphi)
  + \sigma\epsilon (\nabla \varphi^{k+1}_h,\nabla (e_\varphi-\Pi_he_\varphi))\\
  =& \sum_{T\in \mathcal{T}^{k+1}}
  \left[
  (r^{(3)}_h,e_\varphi-\Pi_he_\varphi)_T
  +\sigma\epsilon(\nabla \varphi^{k+1}_h,\nabla (e_\varphi-\Pi_he_\varphi))
  \right] \\
  =& \sum_{T\in \mathcal{T}^{k+1}}
  \left[
  (r^{(3)}_h,e_\varphi-\Pi_he_\varphi)_T
  + \sigma\epsilon(\nabla \varphi^{k+1}_h \cdot \nu,e_\varphi-\Pi_he_\varphi)_{\partial T}
  \right] \\
  \leq&\sum_{T\in \mathcal{T}^{k+1}}
  \|r^{(3)}_h\|_T\|e_\varphi-\Pi_he_\varphi\|_T
  +
  \sum_{E\in \mathcal{E}^{k+1}}
  \sigma\epsilon\|\left[\nabla \varphi^{k+1}_h\right]_\nu\|_E\|e_\varphi-\Pi_he_\varphi\|_E \\
  \leq& C\left(
  \sum_{T\in \mathcal{T}^{k+1}}
  \left(\eta_T^{(3)}\right)^2
  +
  (\sigma\epsilon)^2\sum_{E\in \mathcal{E}^{k+1}}
  \left(\eta_E^{(3)}\right)^2
  \right)^{1/2} \|\nabla e_\varphi\|_\Omega.
\end{align*}
Here $C$ is a generic constant.
In the same manner we derive
\begin{align*}
  \eta_1 \leq C\left(
  \sum_{T\in \mathcal{T}^{k+1}}
  \left(\eta_T^{(1)}\right)^2
  +
  \tau^2\sum_{E\in \mathcal{E}^{k+1}}
  \left(\eta_E^{(1)}\right)^2
  \right)^{1/2} \|\nabla e_v\|_\Omega,\\
  \eta_2 \leq C\left(
  \sum_{T\in \mathcal{T}^{k+1}}
  \left(\eta_T^{(2)}\right)^2
  +
  \tau^2\sum_{E\in \mathcal{E}^{k+1}}
  \left(\eta_E^{(2)}\right)^2
  \right)^{1/2} \|\nabla e_\mu\|_\Omega.
\end{align*}
Using Young's inequality we now directly obtain:

\begin{theorem}\label{thm:AD:errorEstimate}
There exists a constant $C>0$ only depending on the domain $\Omega$ and the regularity of the mesh
$\mathcal{T}^{k+1}$ such that
\begin{align*}
  \underline{\rho}\|e_v\|^2
  + \tau \underline\eta \|\nabla e_v\|^2
  +\tau \underline m\|\nabla e_\mu\|^2
  + \sigma\epsilon\|\nabla e_\varphi\|^2
  +(F'_+(\varphi^{k+1}_h)-F'_+(\varphi^{k+1}),e_\varphi)\\
  \leq C\left( \eta_\Omega^2 + \eta_{h.o.t} + \eta_C \right),
\end{align*}
holds with
\begin{align*}
  \eta_\Omega^2 = &
  \frac{1}{\tau\underline \eta}
  \sum_{T\in \mathcal{T}^{k+1}}
  \left(\eta_T^{(1)}\right)^2
  +
  \frac{\tau}{\underline\eta}\sum_{E\in \mathcal{E}^{k+1}}
  \left(\eta_E^{(1)}\right)^2\\
  &
  \frac{1}{\tau\underline m}
  \sum_{T\in \mathcal{T}^{k+1}}
  \left(\eta_T^{(2)}\right)^2
  +
  \frac{\tau}{\underline m}\sum_{E\in \mathcal{E}^{k+1}}
  \left(\eta_E^{(2)}\right)^2\\
  &
  \frac{1}{\sigma\epsilon}
  \sum_{T\in \mathcal{T}^{k+1}}
  \left(\eta_T^{(3)}\right)^2
  +
  \sigma\epsilon\sum_{E\in \mathcal{E}^{k+1}}
  \left(\eta_E^{(3)}\right)^2,\\
  \eta_{h.o.t.} =& \tau (\mbox{div}(e_v),e_p), \\
  \mbox{ and }\eta_C = &(\mathcal{P}^{k+1}\varphi^k-\varphi^k,e_\mu)
  - (F'_-(\mathcal{P}^{k+1}\varphi^k) -  F'_-(\varphi^k),e_\varphi).
\end{align*}
\end{theorem}

\begin{remark}
~
\begin{itemize}
  \item The term $\eta_{h.o.t.}$ is of higher order. By approximation results it can
  be estimated in terms of $h_T$ to a higher order then the orders included in $\eta_T^{(i)}$,
  $\eta_E^{(i)}$, $i=1,2,3$. Thus it is neglected in the numerics.
  \item  The term $\eta_C$ arises due to the transfer of $\varphi^k$ from the old grid
  $\mathcal{T}^k$ to the new grid $\mathcal{T}^{k+1}$ through the projection
  $\mathcal{P}^{k+1}$.
  In our numerics presented in Section \ref{sec:Numerics} we use Lagrangian interpolation
  $\mathcal  I^{k+1}$ as projection operator.     We   note that
  $\mathcal I^{k+1}\varphi^k$ and $\varphi^k$ do only differ in regions of the domain
  where coarsening in the last time step took place, if bisection is used as refinement strategy.
  Since it seems unlikly that elements being coarsend in the last time step are refined
  again in the present time step, this term is
  neglected in the numerics. We note that this term might be further estimated to obtain powers of
  $h_T$ by approximation results for the Lagrange interpolation, see e.g. \cite{ErnGuermond}.
  \item Due to these two terms involved the estimator is not fully reliable.
  \item Neglecting these two terms the estimator can be shown to be efficient by the  standard
  bubble technique, see e.g. \cite{HintermuellerHinzeTber,AinsworthOden_Aposteriori}.
  \item An adaptation of the time step size is not considered in the present work, since it would
  conflict with the time discretization over three time instances.
  In our numerics we have to choose  time steps small enough to sufficently well resolve the
  interfacial force $\mu^{k+1}_h\nabla \varphi^k$.
\end{itemize}
\end{remark}

In the numerical part, this error estimator is used together with the mesh adaptation cycle
described in \cite{HintermuellerHinzeTber}.
The overall adaptation cycle
\begin{center}
SOLVE $\to$ ESTIMATE $\to$ MARK $\to$ ADAPT
\end{center}
is performed once per time step.
For convenience of the reader we state the marking strategy here.
\begin{myAlgorithm}[Marking strategy]
~\label{alg:AD:marking}
\begin{itemize}
  \item Fix $a_{\min}>0$ and $a_{\max}>0$, and set $\mathcal{A} = \{T\in \mathcal{T}^{k+1}\,|\,
  a_{\min}\leq |T|\leq a_{\max}\}$.
  \item Define indicators:
  \begin{enumerate}
    \item $\eta_T = \frac{1}{\tau \underline \eta}\left(\eta_T^{(1)}\right)^2
    + \frac{1}{\tau \underline m}\left(\eta_T^{(2)}\right)^2
    + \frac{1}{\sigma\epsilon}\left(\eta_T^{(3)}\right)^2$,
    \item $\eta_{TE} =\sum_{E\subset T}\left[
    \frac{\tau}{\underline \eta}\left(\eta_{TE}^{(1)}\right)^2
    +\frac{\tau}{\underline m}\left(\eta_{TE}^{(2)}\right)^2
    +\sigma\epsilon \left(\eta_{TE}^{(3)}\right)^2\right]$.
  \end{enumerate}
  \item Refinement: Choose $\theta^r\in (0,1)$,
  \begin{enumerate}
    \item Find a set $R^T\subset\mathcal{T}^{k+1}$ with
    $\theta^r\sum_{T\in \mathcal{T}^{k+1}}\eta_T\leq \sum_{T \in R^T}\eta_T$,
    \item Find a set $R^{TE}\subset\mathcal{T}^{k+1}$ with
    $\theta^r\sum_{T\in \mathcal{T}^{k+1}}\eta_{TE}\leq \sum_{T \in R^{TE}}\eta_{TE}$.
  \end{enumerate}
  \item Coarsening: Choose $\theta^c\in (0,1)$,
  \begin{enumerate}
    \item Find the set $C^T\subset\mathcal{T}^{k+1}$ with
    $\eta_T\leq \frac{\theta^c}{N}\sum_{T \in \mathcal{T}^{k+1}}\eta_T\,\forall T\in C^T$,
    \item Find the set $C^{TE}\subset\mathcal{T}^{k+1}$ with
    $\eta_{TE}\leq \frac{\theta^c}{N}\sum_{T \in \mathcal{T}^{k+1}}\eta_{TE}\,\forall T\in C^{TE}$.
  \end{enumerate}
  \item Mark all triangles of $\mathcal{A}\cap (R^T\cup R^{TE})$ for refining.
  \item Mark all triangles of $\mathcal{A}\cap (C^T\cup C^{TE})$ for coarsening.
\end{itemize}
\end{myAlgorithm}

\subsection*{Ensuring the validity of the energy estimate}
To ensure the validity of the energy estimate during the numerical computations we  ensure
that Assumption \ref{ass:FP_leq_F} holds trianglewise.
For the following considerations we restrict to bisection as refinement strategy combined with the
$i$FEM coarsening strategy proposed in \cite{iFEMpaper}. This strategy only coarsens patches
consisting of four triangles by replacing them by two triangles if the central node of the patch is
an inner node of $\mathcal{T}^{k+1}$, and patches consisting of two triangles by replacing them by
one triangle if the central node of the patch lies on the boundary of $\Omega$. A patch fulfilling
one of these two conditions we call a nodeStar.
By using this strategy,  we do not harm the Assumption \ref{ass:FP_leq_F} on triangles that are
refined.
We note that this assumption can
only be violated on patches of triangles where coarsening appears.

After marking triangles for refinement and coarsening  and before applying refinement and
coarsening to $\mathcal{T}^{k+1}$ we make a postprocessing of all triangles that are marked for
coarsening.

Let $M^C$ denote the set of triangles marked for coarsening obtained by the marking strategy
described in Algorithm \ref{alg:AD:marking}.
To ensure the validity of the energy estimate \eqref{eq:FD:energyInequality} we perform the
following post processing steps:
\begin{myAlgorithm}[Post processing]
~\label{alg:AD:PostProcessing}
\begin{enumerate}
  \item For each triangle $T \in M^C$:\\
  if $T$ is not part of a nodeStar\\
  then set $M^C := M^C\setminus T$.
  \item For each nodeStar $S \in M^C$:\\
  if Assumption \ref{ass:FP_leq_F} is not fulfilled on $S$\\
  then set $M^C := M^C \setminus S$.
\end{enumerate}

\end{myAlgorithm}

The resulting set $M^C$ does only contain triangles yielding nodeStars on which the Assumption
\ref{ass:FP_leq_F} is fulfilled.

%%%%%%%%%%%%%%%%%%%%%%%%%%%%%%%%%%%%%%%%%%%%%%%%%%%%%%%%%%%%%%%%%%%%%%%%%%%%%%%%%%%
%%%%%%%%%%%%%%%%%%%%%%%%%%%%%%%%%%%%%%%%%%%%%%%%%%%%%%%%%%%%%%%%%%%%%%%%%%%%%%%%%%%
%%%%%%%%%%%%%%%%%%%%%%%%%%%%%%%%%%%%%%%%%%%%%%%%%%%%%%%%%%%%%%%%%%%%%%%%%%%%%%%%%%%
%%%%%%%%%%%%%%%%%%%%%%%%%%%%%%%%%%%%%%%%%%%%%%%%%%%%%%%%%%%%%%%%%%%%%%%%%%%%%%%%%%%
%Here we show the behaviour if the concept:
%(1) validity of inequality by showing data for spinDec
%(2) Mesh for BM1 with and without postprocessing
%(3) deviation of mass
%(4) Benchmark data
% section identifier: EX
\section{Numerics}\label{sec:Numerics}
Now we use the adaptive concept developed in Section \ref{sec:Adaptivity}  to investigate the
evolution of the energy inequality on the numerical level.

The nonlinear system \eqref{eq:AD:chns1_fullDisc}--\eqref{eq:AD:chns4_fullDisc} appearing in every
time step of our approach is solved using the semi-smooth Newton method.
Let us first describe how the linear systems arising in Newton's method are
solved.
At each time step in the Newton iteration we have to solve systems with linear operators $G$ of the
form
\begin{align*}
  G =
  \left(
  \begin{array}{c|c}
\mathcal F  & \mathcal I\\
\hline
\mathcal T & \mathcal C
\end{array}
\right)
=
\left(
\begin{array}{cc|cc}
A   & B     & I     & 0\\
B^t & 0     & 0     & 0\\
\hline
T   & 0     & C_{11}& C_{12}\\
0   & 0     & C_{21}& C_{22}
\end{array}
\right)
.
\end{align*}
Here $\mathcal F$ and $\mathcal C$ are the discrete realizations of  linearized Navier--Stokes and
Cahn--Hilliard systems, respectively,
while $\mathcal I$ represents their  coupling through the interfacial force, and $\mathcal T$  the
coupling through the transport at the interface. The order of the unknowns is $(v,p,\mu,\varphi)$. 

Unique solvability of the systems arising from Newton's method can be shown by using the energy
method of Section \ref{sec:fullyDiscrete} taking Assumption \ref{ass:F_NewtonDiff} into account.

The system is solved by a preconditioned gmres iteration with restart after 10 iterations. As
preconditioner we use the block diagonal preconditioner
\begin{align*}
  \mathcal{P} =
  \begin{pmatrix}
\tilde {\mathcal F} & 0\\
0 & \mathcal C
\end{pmatrix}
\end{align*}
where $\mathcal C$ is inverted by LU decomposition,
while $\tilde {\mathcal  F}$ is an upper triangular block preconditioner
(\cite{BramblePasciak_TriangularPreconditioner}) for Oseen type problems. It uses the
$F_p$ preconditioner \cite{kayLoghinWelford_FpPreconditioner} for the Schur complement, i.e.
\begin{align*}
  \tilde  {\mathcal F}=
  \begin{pmatrix}
\tilde A & B \\
0 & \tilde S
\end{pmatrix},
\end{align*}
where $\tilde S$ is the $F_p$ preconditioner for the Schur complement of $\mathcal F$ and $\tilde
A$ is composed of the diagonal blocks of $A$ and is inverted by LU decomposition.

The implementation is done in C++, where the adaptive concept is build upon $i$FEM
(\cite{iFEMpaper}). As linear solvers we use {\bf umfpack} (\cite{UMFPACK}) and {\bf cholmod}
(\cite{Cholmod}). The Newton iteration is implemented in its inexact variant, ensuring local
superlinear convergence.

\subsection*{Examples}
We investigate the evolution of the free energy and the validity of the energy
inequality. Since we use Lagrange interpolation as projection operator, we violate the conservation
of mass whenever coarsening is performed. This is numerically investigated.

Thereafter we give
results for a qualitative benchmark for rising bubble dynamics.
For this example we  also show the influence of the required post processing step concerning the
evolution of the meshes.

Concerning the free energy $F$ we use the relaxed double-obstacle free energy \eqref{eq:freeEnergy}
and set the relaxation parameter to $\hp = 10000$.

\subsubsection*{Investigation of the free energy}
We start by investigating the evolution of the free energy and the validity of the energy inequality
in Theorem \ref{thm:FD:energyInequality}.
Here we use the classic example of spinodal
decomposition \cite{CahnHilliard,whatSpinodalDecomposition} as test case.
The parameters are chosen as: $\rho_1=\rho_2 = \eta_1 = \eta_2 = 1$, $g \equiv 0$, and
$m(\varphi) \equiv 10^{-3}\epsilon$, $\epsilon=0.01$, $\sigma = 0.01$ and $\tau = 10^{-5}$.

In absence of outer forces the spinodal decomposition admits a characteristic speed of demixing,
see e.g. \cite{Siggia,otto_seis__crossover_of_demixing_rates}. Especially in the case of a diffusion
driven setting the Ginzburg--Landau energy $E$ decreases with the rate $E \sim t^{-1/3}$.

In Figure \ref{fig:EX:spinDec_Energy} we show the time evolution of the  monotonically decreasing
Ginzburg--Landau energy (left plot).
We obtain the expected rate of $E \sim t^{-1/3}$ and also observe a time span where $E \sim t^{-1}$
holds, as predicted in \cite{otto_seis__crossover_of_demixing_rates}.

Next we investigate the validity of the energy inequality, see Figure \ref{fig:EX:spinDec_Energy}
(right plot).
We there show  the time evolution of the term
\begin{align*}
  \zeta = &\frac{1}{2}\int_\Omega \rho^k\left(v^{k+1}_h\right)^2\,dx
  + \frac{\sigma\epsilon}{2}\int_\Omega |\nabla \varphi^{k+1}_h|^2\,dx
  +\int_\Omega F(\varphi^{k+1}_h)\,dx\nonumber\\
  &+ \frac{1}{2}\int_\Omega \rho^{k-1}(v^{k-1}_h-v^k)^2\,dx
  + \frac{\sigma\epsilon}{2}
  \int_\Omega |\nabla \varphi^{k+1}_h-\nabla\mathcal{I}^{k+1} \varphi^k|^2\,dx\nonumber\\
  & + \tau\int_\Omega 2\eta^k|Dv^{k+1}_h|^2\,dx
  +\tau\int_\Omega m^k|\nabla \mu^{k+1}_h|^2\,dx \nonumber\\
  & -
  \left(\frac{1}{2}\int_\Omega \rho^k\left(v^{k}\right)^2\,dx
  + \frac{\sigma\epsilon}{2}\int_\Omega |\nabla \mathcal{I}^{k+1}\varphi^{k}|^2\,dx
  +\int_\Omega F(\mathcal{I}^{k+1}\varphi^{k})\,dx  + \int_\Omega \rho^k g v^{k+1}_h\right).
\end{align*}
The post processing of Algorithm \ref{alg:AD:PostProcessing} guarantees, that this term is always
negative.
The influence of Algorithm \ref{alg:AD:PostProcessing} on the mesh quality is investigated later.

\begin{figure}
  \centering
  \includegraphics[width=0.4\textwidth]{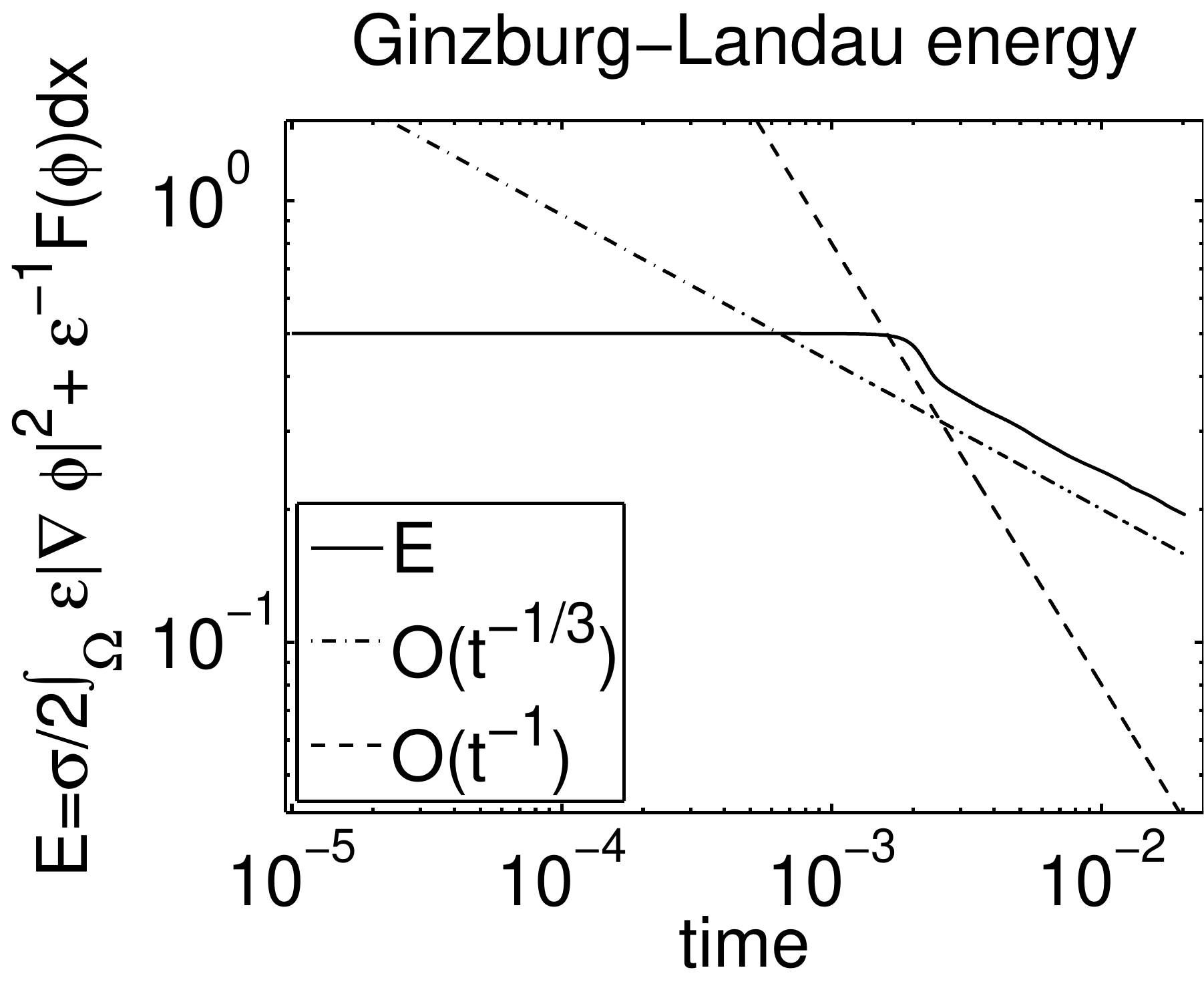}
  \hspace{2em}
  \includegraphics[width=0.4\textwidth]{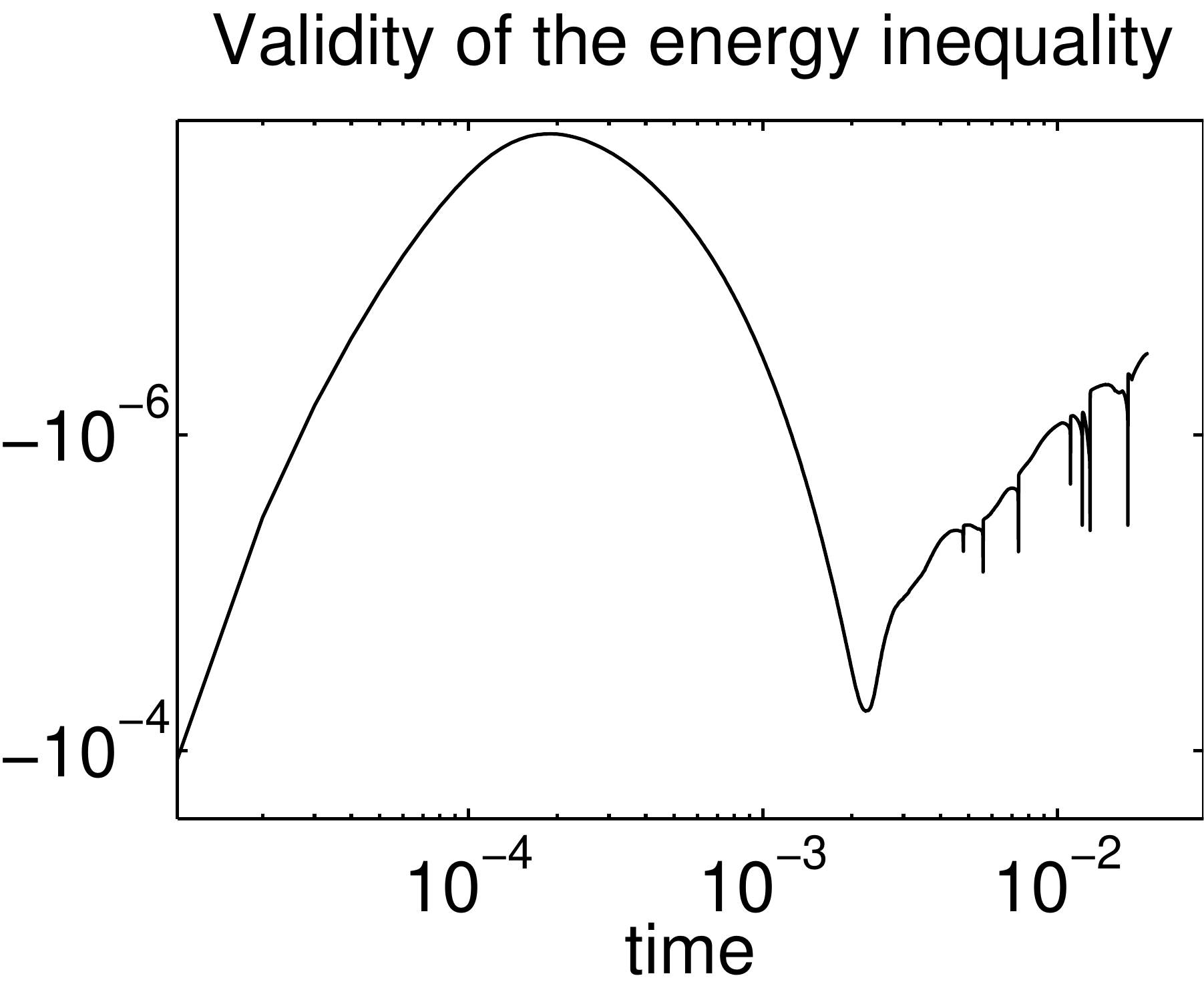}
  \caption{Time evolution of the Ginzburg--Landau energy
  (left), and validity of the energy inequality (right).}
  \label{fig:EX:spinDec_Energy}
\end{figure}

\subsubsection*{The violation in the conservation of mass}
Since we use Lagrange interpolation as projection operator between successive grids, we do not
have full mass conservation, but have a violation in the mean value of $\varphi$ as discussed in
Remark \ref{rm:ProlongationMassDeviation}. In Figure \ref{fig:EX:spinDec_MassDeviation} we depict
the time evolution of the term $\left|\int_\Omega \varphi^{k+1} - \varphi^0\,dx\right|$, i.e. the
difference between the mean value of $\varphi$ and the mean value of the initial phase field $\varphi^0$. The
numerical setup is the spinodal decomposition.

\begin{figure}
  \centering
  \includegraphics[width=0.4\textwidth]{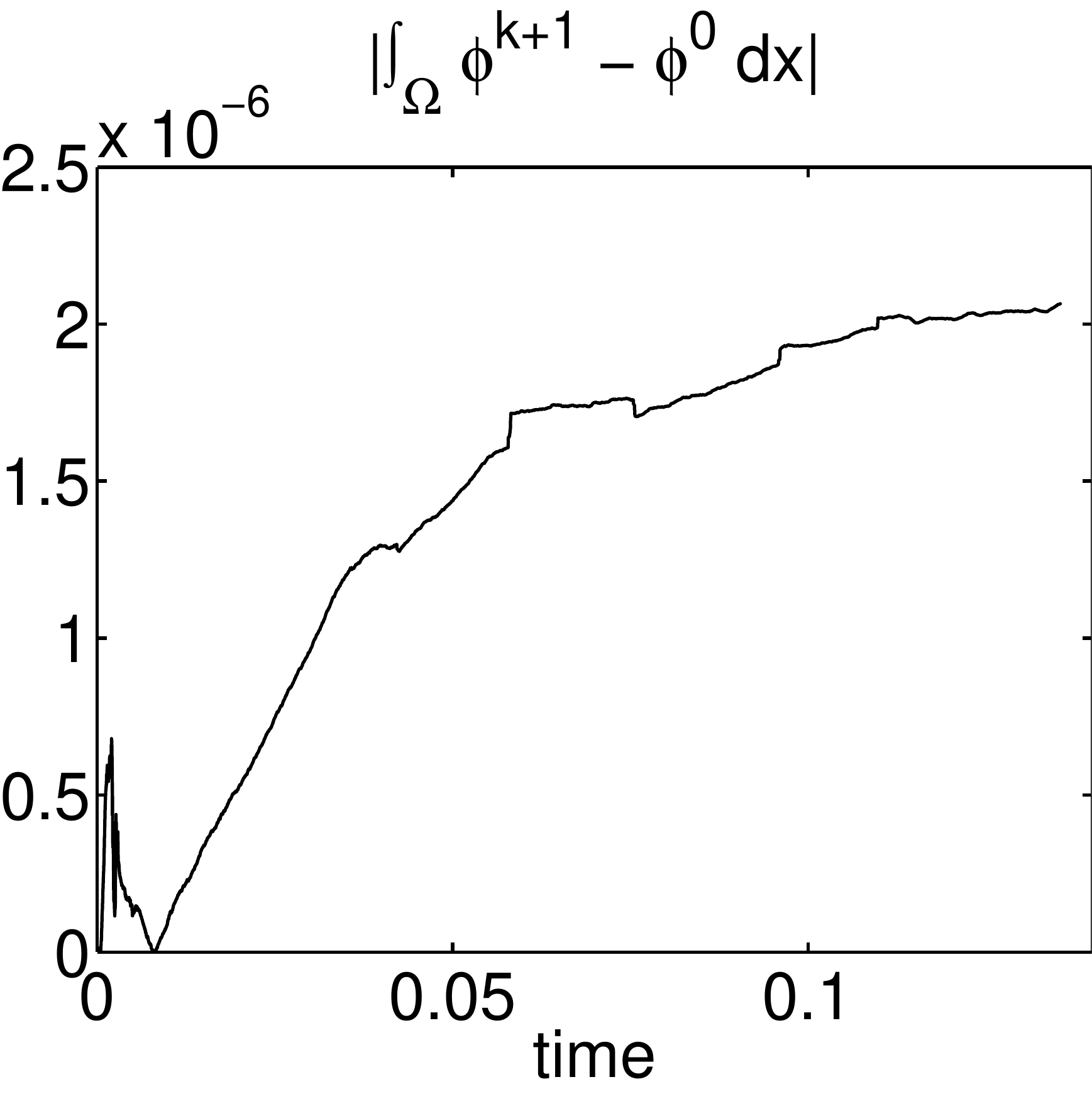}
  \caption{Time evolution of the deviation of the mean value of $\varphi$.}
  \label{fig:EX:spinDec_MassDeviation}
\end{figure}

As can be observed, the violation increases with time, and the violation in mass conservation
finally is of size $10^{-6}$.
We note that the order of the mean value is $|\Omega|$
and here we have $|\Omega|=1$. Thus though we have deviation of mass, its size is small in
comparison to the actual mean value.

\subsubsection*{Comparison with an existing benchmark}
In \cite{Hysing_Turek_quantitative_benchmark_computations_of_two_dimensional_bubble_dynamics}
a quantitative benchmark for the simulation of the rising bubble is proposed. Three different
groups provided numerical results for two benchmark computations of rising bubble scenarios, using
sharp interface models. In \cite{Aland_Voigt_bubble_benchmark} the
\cite{Hysing_Turek_quantitative_benchmark_computations_of_two_dimensional_bubble_dynamics}
benchmark is implemented with computations based on three different diffuse interface approximations
to the setup of
\cite{Hysing_Turek_quantitative_benchmark_computations_of_two_dimensional_bubble_dynamics}.

We briefly describe the setup. We start with an inital bubble of radius $r=0.25$ located at $M =
(0.5,0.5)$ with physical surface tension $\sigma^{phys} = 24.5$, resulting in $\sigma  \approx
15.5972$ (see \cite[Sec. 4.3.4]{AbelsGarckeGruen_CHNSmodell}).
The initial velocity is zero.
In the domain $\Omega = (0,1)\times (0,2)$ we have no-slip boundary conditions for the
velocity field on top and bottom walls, and free-slip on the left and the right walls.
The parameters are given as $\rho_1 = 1000$, $\rho_2 = 100$, $\eta_1 = 10$, $\eta_2 = 1$, resulting
in a Reynolds number of 35.
Here $\rho_2$ and $\eta_2$ correspond to the fluid in the bubble. Due to the smaller
density we expect the bubble to rise in the gravity field with force $g = (0,-0.98)^t$.
Since we use a diffuse interface model, we have
the additional parameters $m \equiv 10^{-3}\epsilon$ and $\epsilon = 0.02$. The time
discretization step is chosen as $\tau = 2.5e-5$.
The rising bubble is simulated over a time horizon of $3$ units
of time.

\bigskip

In \cite{Hysing_Turek_quantitative_benchmark_computations_of_two_dimensional_bubble_dynamics}
the following benchmark parameters are defined.
For a bubble represented  by $\varphi(x)<0$ we  measure the evolution of
circularity, rising velocity and of the  center of mass.

The circularity is defined by
\begin{equation*}
  \Theta_\varphi =
  \frac{\mbox{perimeter\,of\,area-equivalent\,circle}}{\mbox{perimeter\,of\,bubble}}\leq 1,
\end{equation*}
the rising velocity is defined as
\begin{equation*}
  V_\varphi = \frac{\int_{\varphi<0} v \, dx}{\int_{\varphi<0} 1 \, dx},
\end{equation*}
and the center of mass is given by
\begin{equation*}
  M_\varphi = \frac{\int_{\varphi<0} x_2 \, dx}{\int_{\varphi<0} 1 \, dx}.
\end{equation*}
Here $x_2$ denotes the second component of the spatial variable $x = (x_1,x_2)$. Note that the
process is symmetric and it is sufficient to integrate over the second component.

As benchmark values the minimal circularity  $(\Theta_\varphi)_{\min}$
together with the time
$t_\Theta := t(\Theta_\varphi \equiv (\Theta_\varphi)_{\min})$,
the maximal rising velocity
$(V_\varphi)_{\max}$
together with the time
$t_V := t(V_\varphi \equiv (V_\varphi)_{\max})$
and the center of mass $M_\varphi(t=3)$ at the final time $t=3$ are presented.

Our results are shown in Table \ref{tab:EX:BM1}, first row (GHK $\epsilon=0.02$). For comparison
we in the second row also give the results obtained in \cite{Aland_Voigt_bubble_benchmark}  (AV
$\epsilon = 0.02$).
Note that there the polynomial double well free energy is used resulting in a diffuser interface.
The results in the following rows are taken from the sharp interface numerics in
\cite{Hysing_Turek_quantitative_benchmark_computations_of_two_dimensional_bubble_dynamics}.
The
groups TP2D, FreeLIFE and MooNMD are the groups participating in
\cite{Hysing_Turek_quantitative_benchmark_computations_of_two_dimensional_bubble_dynamics}. The
group TP2D is the group of Turek at the Technical University of Dortmund, FreeLIFE is provided by
the \'Ecole polytechnique f\'ed\'erale de Lausanne by the group of Burman and MooNMD is the group of
Tobiska from the University of Magdeburg.
With BGN we denote the results presented in
\cite{BarrettGarckeNuernberg_stableParametricFEdisc}, which are obtained by a sharp interface
approach based on the model used in the present paper.
\begin{table}
\centering
\begin{tabular}{cccccc}
\hline
Group &
$(\Theta_\varphi)_{\min}$ & $t_\Theta$ &
$(V_\varphi)_{\max}$ & $t_V$ &
$M_\varphi(t=3)$\\
\hline
GHK $\epsilon=0.02$  & 0.9080 & 1.9672 & 0.2388 & 0.9765 & 1.0786\\
\hline
AV  $\epsilon=0.02$  & 0.9159 & 2.0040 & 0.2375 & 1.0400 & 1.0733\\
\hline
\hline
TP2D     & 0.9013 & 1.9041 & 0.2417 & 0.9213 & 1.0813\\
FreeLIFE & 0.9011 & 1.8750 & 0.2421 & 0.9313 & 1.0799\\
MooNMD   & 0.9013 & 1.9000 & 0.2417 & 0.9239 & 1.0817\\
\hline
BGN	     & 0.9014 & 1.9000 & 0.2417 & 0.9230 & 1.0819
\end{tabular}
\caption{Results for the first benchmark from
\cite{Hysing_Turek_quantitative_benchmark_computations_of_two_dimensional_bubble_dynamics}.}
\label{tab:EX:BM1}
\end{table} 

We see that our results are in quite good aggrement with those obtained with sharp interface
numerics.

In Figure \ref{fig:EX:benchmark_evolution} we show the evolution of the bubble for the benchmark
setting.

\begin{figure}
  \centering
  \fbox{\includegraphics[width=0.18\textwidth]{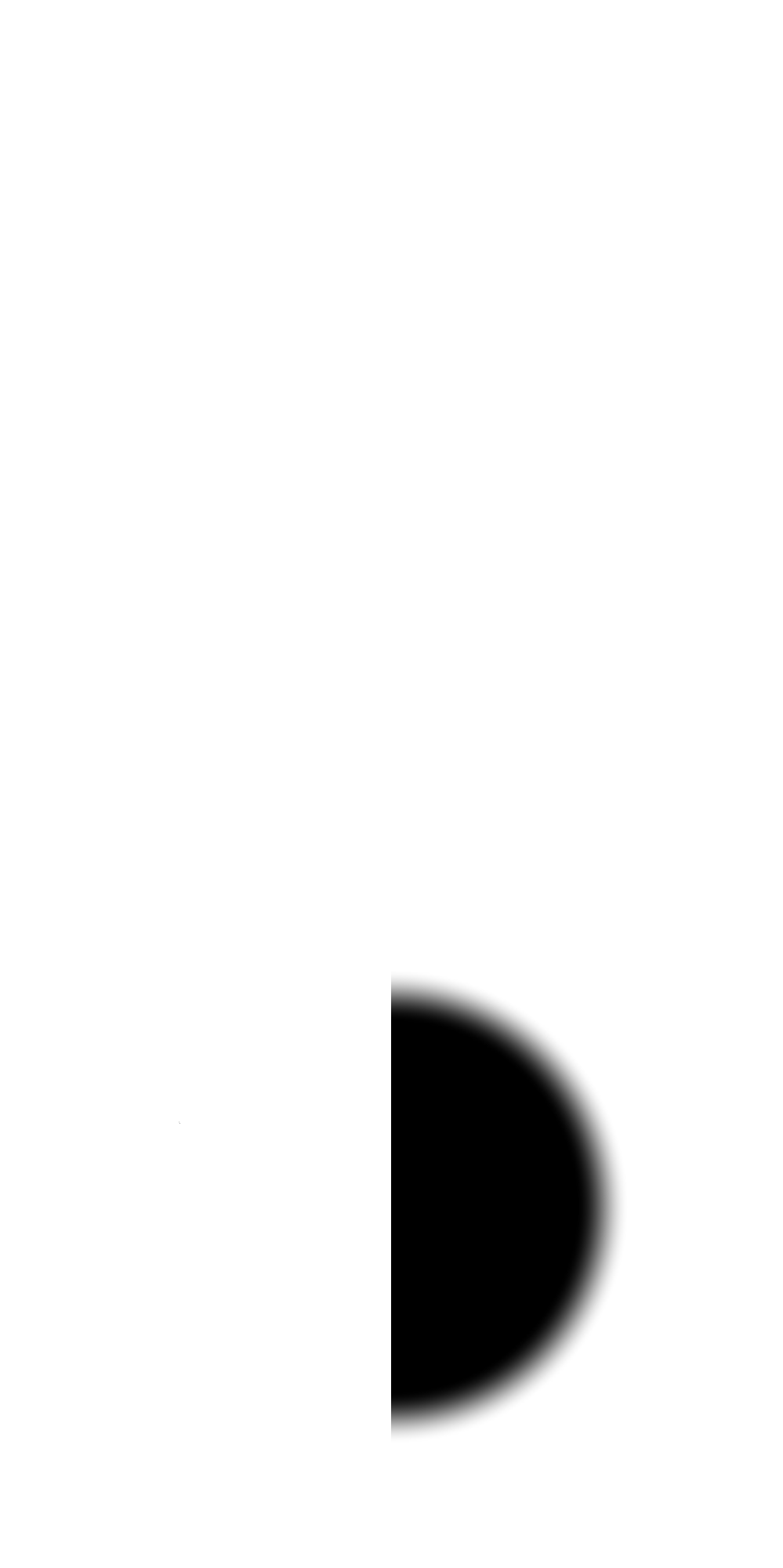}}
  \fbox{\includegraphics[width=0.18\textwidth]{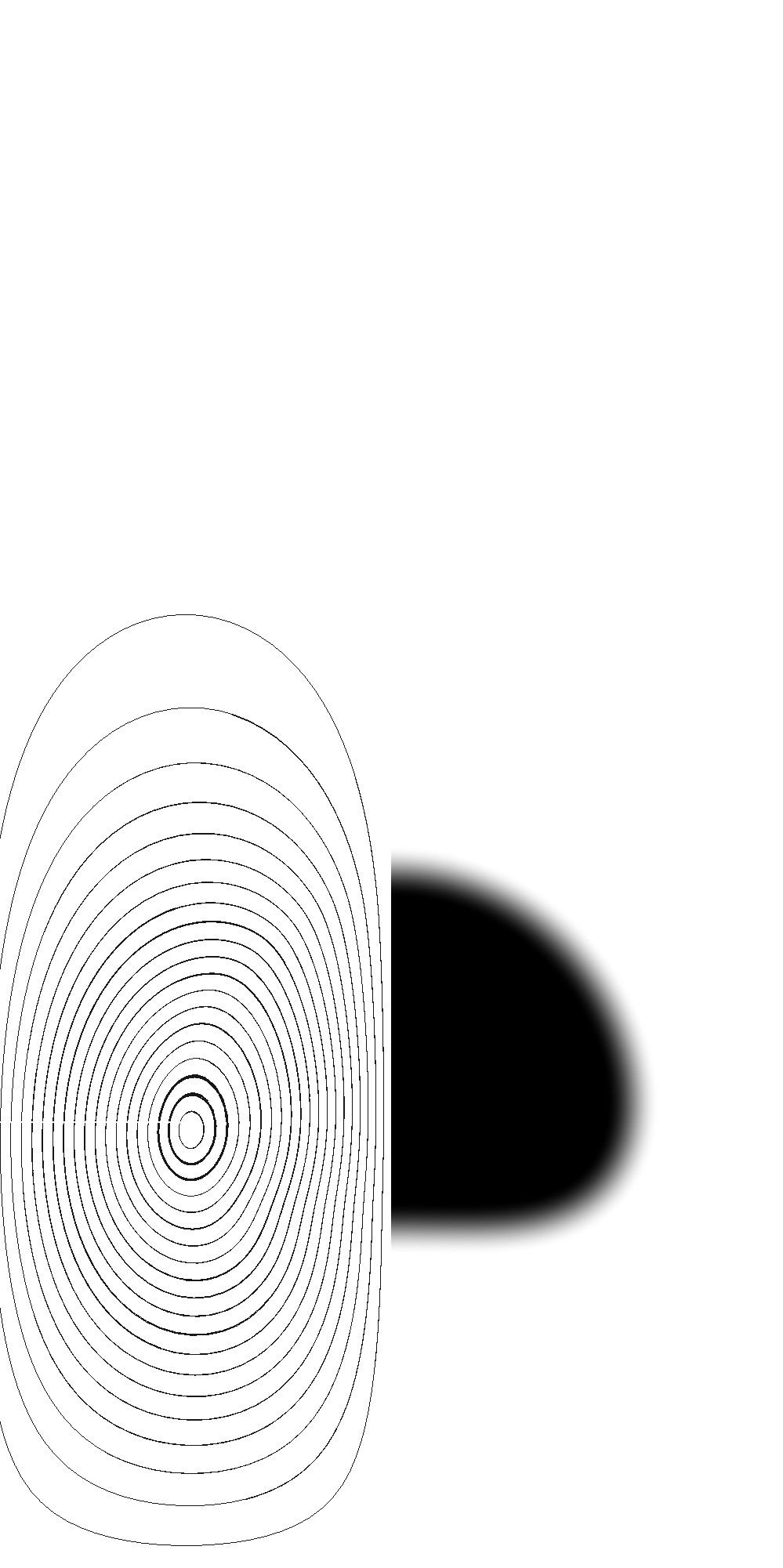}}
  \fbox{\includegraphics[width=0.18\textwidth]{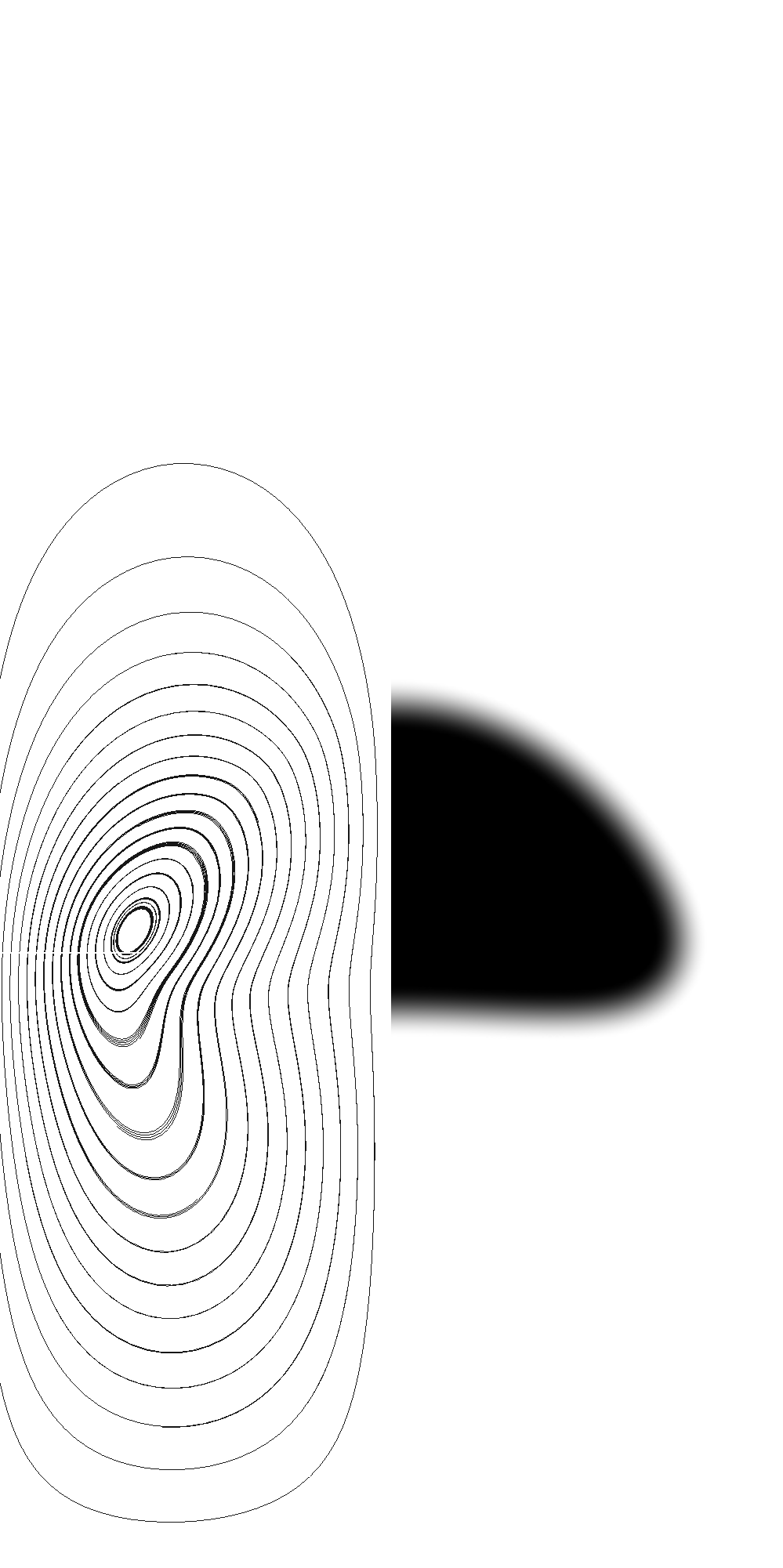}}
  \fbox{\includegraphics[width=0.18\textwidth]{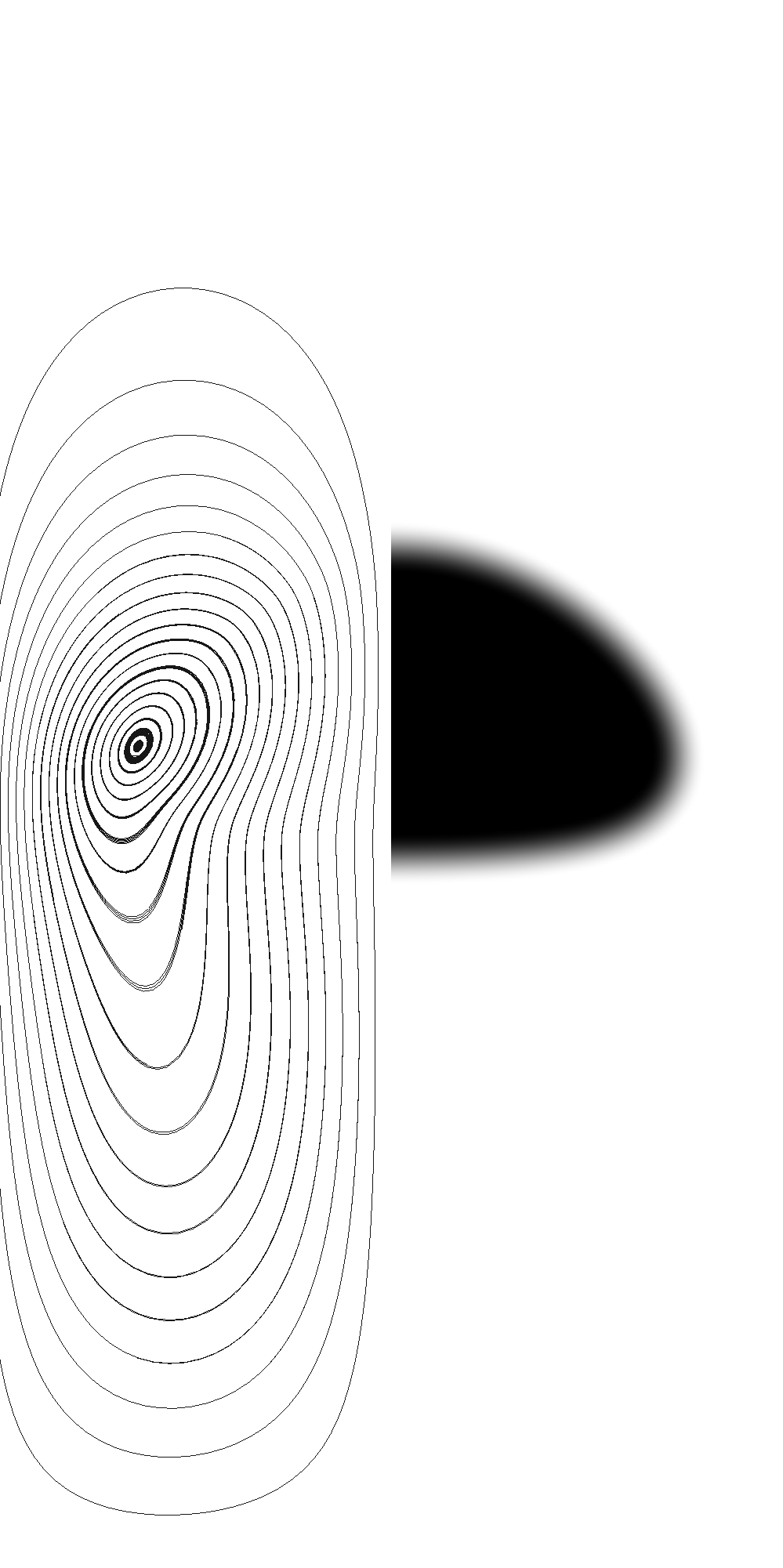}}
  \caption{The evolution of the bubble at times $t\in\{0,1,2,3\}$.
  The phase field is shown in the right part   and streamlines of the velocity field in the left
  part of each plot.}
  \label{fig:EX:benchmark_evolution}
\end{figure}

\subsubsection*{Distribution of the error indicators}
Next we investigate the distribution of the error indicators.
We observe that a similar distribution is observed as in  the
case of the numerical simulation of the Cahn--Hilliard equation with transport reported in
\cite{HintermuellerHinzeKahle_adaptiveCHNS}.
The errors are concentrated at the boundary of the interface.
We further have additional error contributions from the Navier--Stokes part
in a neighborhood of the bubble.

In Figure \ref{fig:EX:error_distribution} we show the distribution of the error indicators
$\eta_T$ and $\eta_{TE}$ defined in Algorithm \ref{alg:AD:marking}.

\begin{figure}
  \centering
\begin{minipage}[t]{0.45\textwidth}
   \centering
  \fbox{\includegraphics[width=0.45\textwidth]{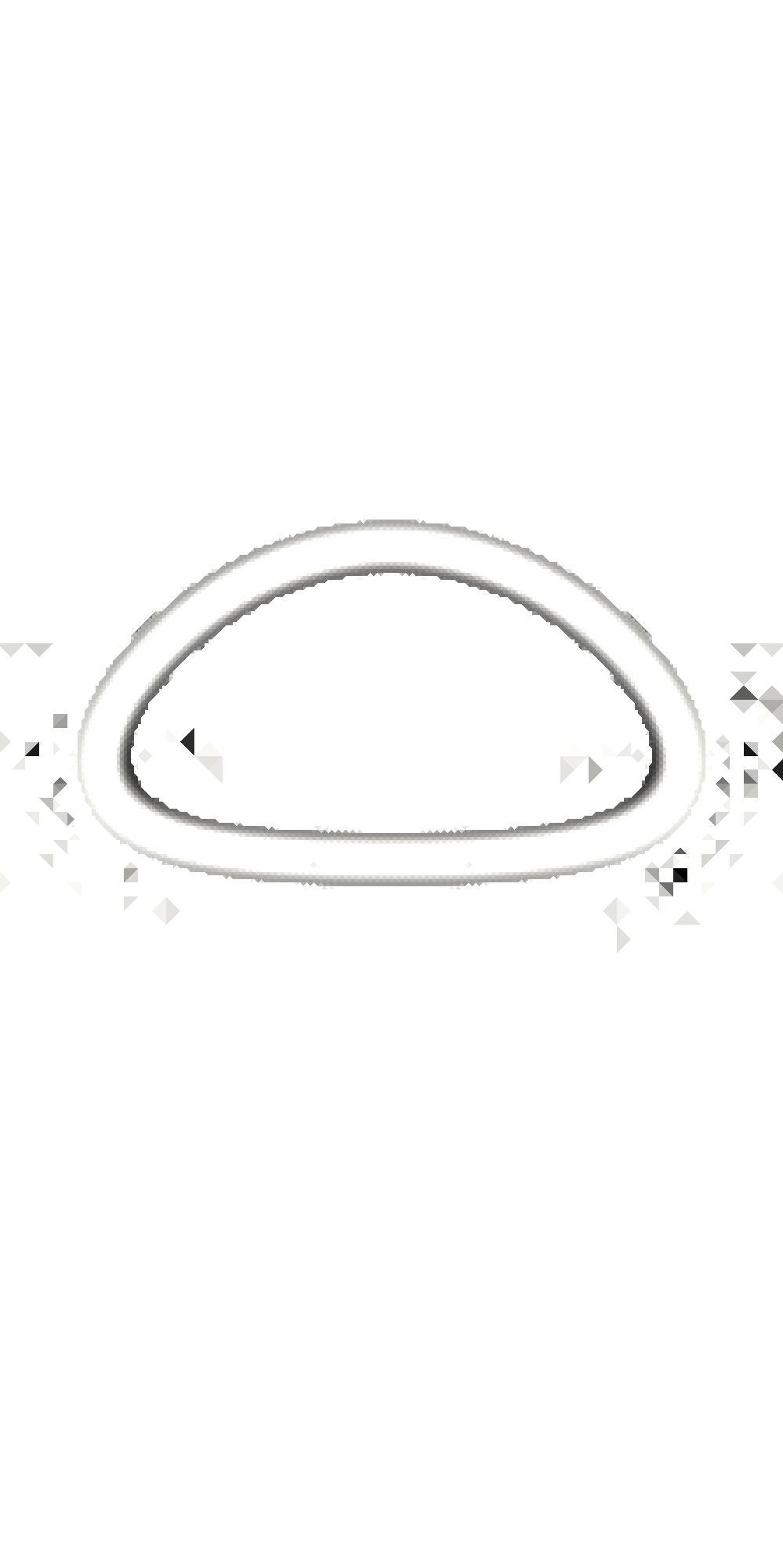}}
  \fbox{\includegraphics[width=0.45\textwidth]{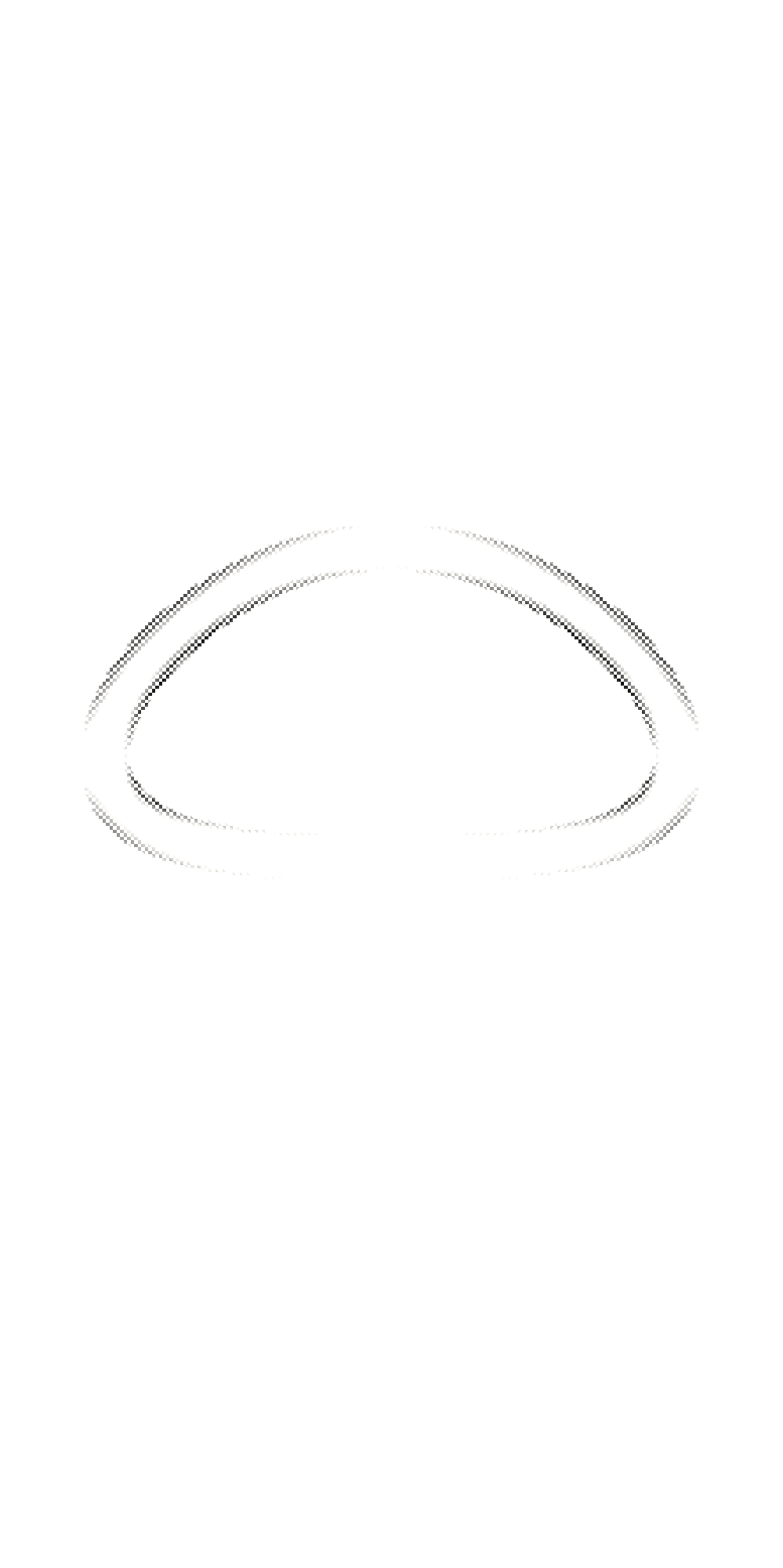}}
  \caption{The distribution of the error indicators at time $t=3$. $\eta_T$ on the left, $\eta_{TE}$
  on the right. Black indicates higher errors.}
  \label{fig:EX:error_distribution}
\end{minipage} 
\hspace{1em}
\begin{minipage}[t]{0.45\textwidth}
\centering
\fbox{\includegraphics[width=0.45\textwidth]{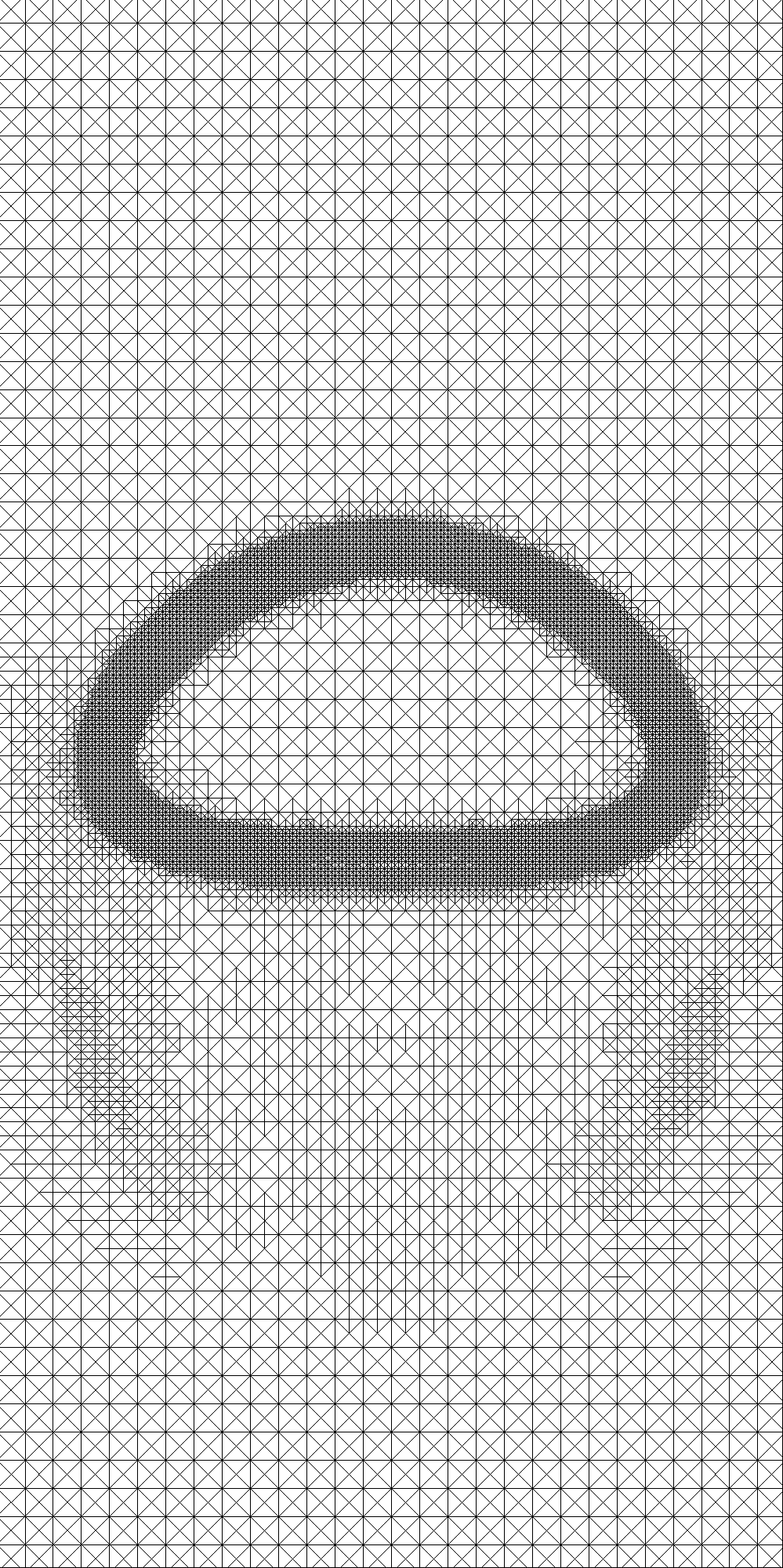}}
\fbox{\includegraphics[width=0.45\textwidth]{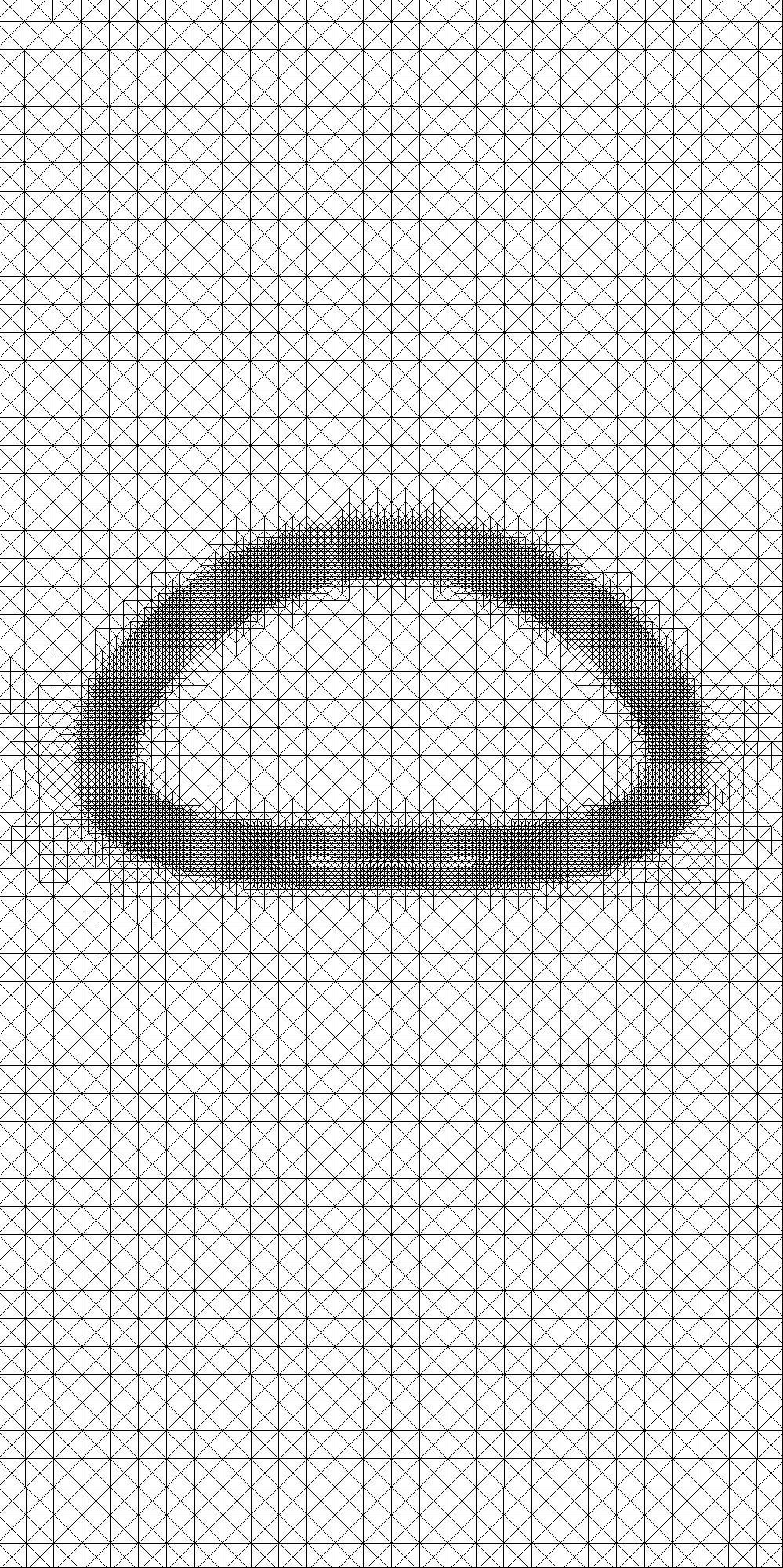}}
\caption{The mesh with (left) and without (right) postprocessing at final time $t=3$.}
\label{fig:EX:finalMesh}
\end{minipage}
\end{figure}

\subsubsection*{Influence of the post processing of the marked triangles}
Finally we investigate the spatial discretization obtained by our adaptive concept. Especially
we show the influence of the post processing step of Algorithm \ref{alg:AD:PostProcessing}
on reducing the number of triangles that are coarsened.

We simulate the rising bubble benchmark in the setting described above with and without the
postprocessing steps.
We note that without the postprocessing
artificial energy is generated numerically through the coarsening process and
 the validity of the energy inequality can not be guaranteed, and
in fact is not given.

In Figure \ref{fig:EX:finalMesh} we show the final meshes at $t=3$ with
postprocessing (left) and without postprocessing (right).
We see that there are regions in
the bulk phase below the bubble where the postprocessing prevents the adaptive strategy from
coarsening the triangles to the coarsest level.
Thus we obtain a larger number of nodes if we use
the post processing as is demonstrated in Figure \ref{fig:EX:NT_evolution} where we display the
evolution of the number of mesh nodes with and without postprocessing.

\begin{figure}
\centering
\includegraphics[width=0.2\textwidth]{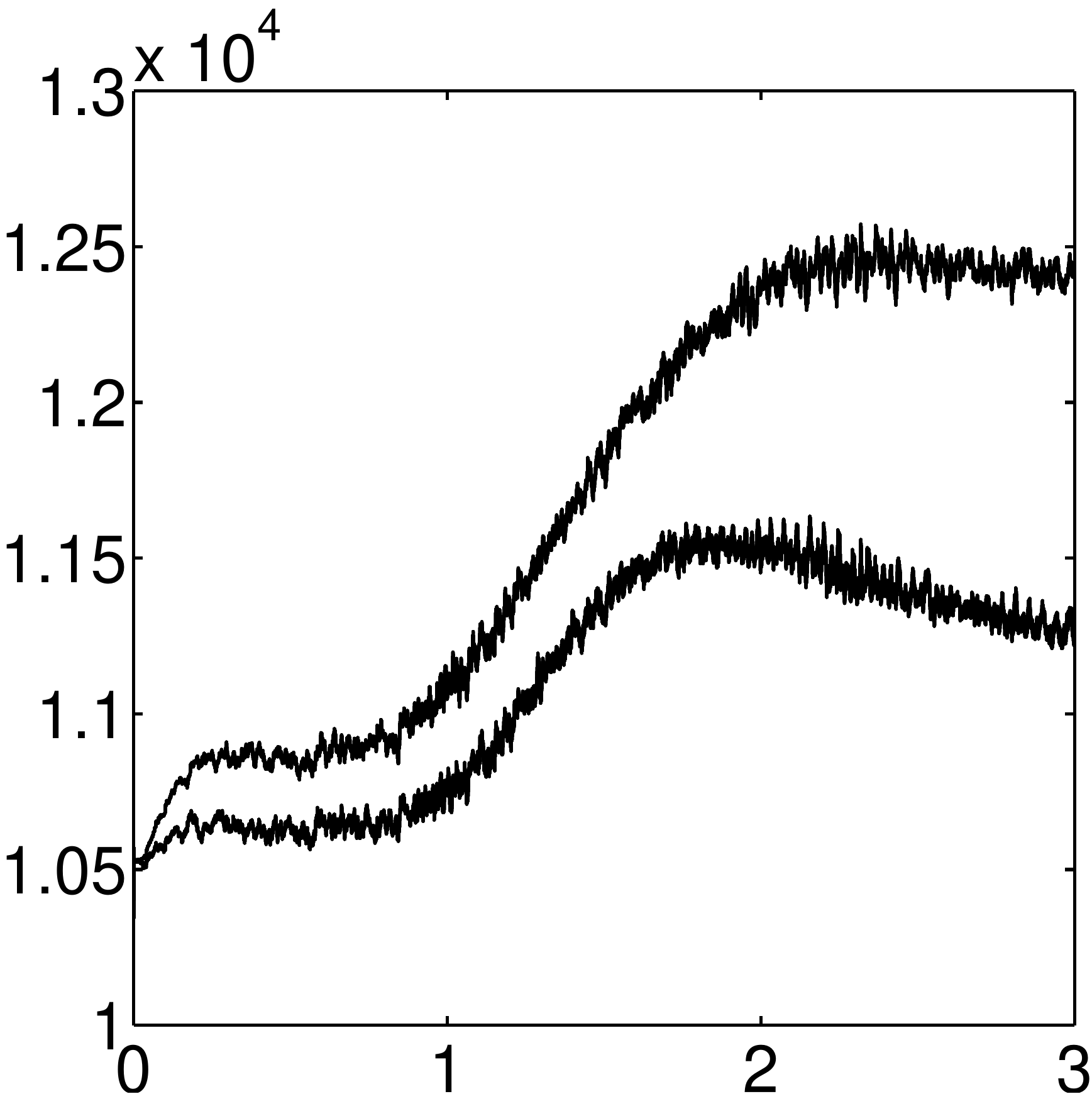}
\caption{The evolution of the number of nodes for the described benchmark with (upper line) and
without (lower line) postprocessing.}
\label{fig:EX:NT_evolution}
\end{figure}
 
We see that the number of nodes increases (by maximal 10\% in this example) since not
all triangles that are marked for coarsening are coarsened.
On the other hand we note, that the
energy inequality in the case without post processing is violated in 1692 of 60000 simulation steps
and this violation takes  place within the first 7000 time steps.

Let us note, that from the fluid mechanical point of view and if one considers the bubble as an
obstacle in the channel flow, the region detected by the post
processing is the wake, where the fluid is accelerated. Thus we expect a refined flow mesh there.

\section{Conclusion}
We propose a time discretization for the thermodynamically consistent model from
\cite{AbelsGarckeGruen_CHNSmodell} that gives rise to a time discrete energy inequality that can
be conserved in the fully discrete setting. The systems to be solved in the discrete setting are
fully coupled and a concept for handling the linear systems arising from Newton's method is
proposed.

Based on the energy inequality we derive an error estimator both measuring the local
error in the discretization of the velocity field, and in the phase field and the chemical
potential.
We investigate the behavior of our  solver and especially could numerically verify the validity of
the discrete energy inequality.

Post processing, applied to fulfill the energy inequality in the discrete setting, in our
example leads to meshes containing approximatly 10\% more triangles than the meshes obtained without
post processing, but to a more reasonable refinement from the flow-physics point of view.
The numerical results might be further improved by replacing the Lagrange interpolation operator by
e.g. the mass-conserving quasi-interpolation operator introduced in
\cite{Carstensen_QuasiInterpolation}.
This will be subject to future work.

% \bibliographystyle{alpha}
% \bibliography{D:/quellen}

\newcommand{\etalchar}[1]{$^{#1}$}

\end{document}